\documentclass[a4paper,10pt]{article}
\usepackage{a4,amssymb,amsfonts,verbatim,pstricks,pst-plot,times}
\usepackage{yfonts}
\usepackage{enumitem}
\usepackage[frenchb,english]{babel}
\usepackage{amssymb}
\usepackage{amsmath}
\usepackage{fancyhdr}
\usepackage{amssymb,latexsym}
\usepackage{mltex}
\usepackage{amsthm}

\newtheorem{thm}{Theorem}     
     
\newtheorem{lem}{Lemma}[section]
\newtheorem{prop}[lem]{Proposition}     
\newtheorem{rem}{Remark}  

\newtheorem{defn}[lem]{Definition}

\newcommand{\Ekt}{\mathbb{E}(\kappa,\tau)}

\newcommand{\Ric}{\mathrm{Ric}}

\newcommand{\spin}{\mathrm{Spin}}

\newcommand{\C}{\mathbb{C}}

\newcommand{\Ss}{\mathbb{S}}
\newcommand{\HH}{\mathbb{H}}

\newcommand{\R}{\mathbb{R}}

\newcommand{\Chi}{\mathfrak{X}}
\newcommand{\pre}{\Re e}

\newcommand{\Spinc}{\mathrm{Spin^c}}
\newcommand{\spinc}{\mathrm{Spin^c}}

\newcommand{\beqt}{\begin{equation}}  \newcommand{\eeqt}{\end{equation}}
\newcommand{\bal}{\begin{align}}      \newcommand{\eal}{\end{align}}
\newcommand{\ba}{\begin{array}}      \newcommand{\ea}{\end{array}}
\newcommand{\bc}{\begin{center}}     \newcommand{\ec}{\end{center}}
\newcommand{\be}{\begin{enumerate}}  \newcommand{\ee}{\end{enumerate}}
\newcommand{\beq}{\begin{eqnarray}}  \newcommand{\eeq}{\end{eqnarray}}
\newcommand{\beQ}{\begin{eqnarray*}} \newcommand{\eeQ}{\end{eqnarray*}}
\newcommand{\bi}{\begin{itemize}}    \newcommand{\ei}{\end{itemize}}
\newcommand{\bt}{\begin{tabular}}    \newcommand{\et}{\end{tabular}}

\title{ Characterization of hypersurfaces in four dimensional product spaces via two different $\Spinc$ structures}
\author{Roger Nakad and Julien Roth}
\date{\today}

\begin{document}
\maketitle
\begin{abstract}  

    The Riemannian product $\mathbb M_1(c_1) \times \mathbb M_2(c_2)$, where  $\mathbb M_i(c_i)$ denotes the $2$-dimensional space form of constant sectional curvature $c_i \in \mathbb R$, has two different $\Spinc$ structures carrying each a parallel spinor.  The restriction of these two parallel spinor fields to a $3$-dimensional hypersurface $M$ characterizes the isometric immersion of $M$ into $\mathbb M_1(c_1) \times \mathbb M_2(c_2)$. As an application, we prove that totally umbilical hypersurfaces of $\mathbb M_1(c_1) \times \mathbb M_1(c_1)$ and totally umbilical hypersurfaces  of $\mathbb M_1(c_1) \times \mathbb M_2(c_2)$  ($c_1 \neq c_2$)  having a local structure product, are of constant mean curvature. 
\end{abstract} 

%\tableofcontents
\section{Introduction}

Over the past years, the real spinorial ($\spin$ geometry) and the complex spinorial ($\spinc$ geometry) approaches have been used fruitfully to characterize (\cite{Fr, Mo, Rot, blr, brj, blr13, LR, NR12, nr17}  and references therein)  submanifolds of some special ambient manifolds.  These approaches allowed also to study the geometry and  topology of submanifolds and solve naturally some  extrinsic problems. For instance, elementary proofs of the Alexandrov theorem in the Euclidean space \cite{hmz},  in the hyperbolic space \cite{hmr} and in the  Minkowski spacetime \cite{hmr} were obtained (see also \cite{HM13, HM14}). In 2006, O. Hijazi, S. Montiel and F. Urbano \cite{omu} constructed on K\"ahler-Einstein manifolds with positive scalar curvature, a
$\spinc$ structure carrying K\"ahlerian Killing spinors. The restriction of these spinors to minimal
Lagrangian submanifolds provides topological and geometric restrictions on these submanifolds. The authors \cite{nr17, nakadroth, NR12}, and
by restricting $\spinc$ spinors, gave an elementary $\spinc$ proof for a Lawson type correspondence between
constant mean curvature surfaces of $3$-dimensional homogeneous manifolds with $4$-dimensional
isometry group \cite{Da2}. Furthermore, they gave necessary and sufficient geometric conditions to immerse
a $3$-dimensional Sasaki manifold and a complex/Lagrangian surface into the complex projective
space of complex dimension $2$.\\

The main idea behind characterizing hypersurfaces of $\spin$ or $\spinc$ manifolds is the restriction to the hypersurface of a special spinor field (parallel, real Killing, imaginary Killing, K\"ahlerian Killing...). For example, the restriction $\phi$ of a parallel spinor field on a Riemannian $\spin$ or $\spinc$ manifold to an oriented  hypersurface $M$ is a solution of the generalized Killing equation
\begin{eqnarray}
\nabla_X\phi= -\frac 12 \gamma(II X) \phi,
\end{eqnarray}
where $\gamma$ and $\nabla$ are respectively the Clifford multiplication and the $\spin$ or $\spinc$ connection on $M$, the tensor $II$ is the Weingarten tensor of the immersion and $X$ any vector field on $M$. Conversely and in the two-dimensional case, the existence of a generalized Killing $\spin$ spinor field allows to immerse $M$ in $\mathbb R^3$ \cite{Fr}.  This characterization has been  extended  to surfaces of  other $3$-dimensonal (pseudo-) Riemannian manifolds \cite{Mo, Rot, LR11}. Moreover, the existence of a generalized Killing $\spinc$ spinor  on a surface $M$ allows to immerse $M$ in the $3$-dimensional homogeneous manifolds with 4-dimensional isometry group \cite{NR12}.  All these previous results are the geometrical invariant versions of previous  works  on  the  spinorial Weierstrass  representation by R.  Kusner  and  N.  Schmidt, B. Konoplechenko, I. Taimanov and many others (see \cite{kono, sch, ta}). \\

In the three dimensional case,  having a generalized Killing $\spin$ or $\spinc$ spinor  is not sufficient  to  characterize the immersion of $M$ in the desired $4$-dimensional manifold. The problem is that unlike in the $2$-dimensional case,  the spinor bundle of a $3$-dimensional manifold does not decompose into subbundles  of  positive  and  negative  half-spinors. In fact, Morel \cite{Mo} proved that the existence of  a Codazzi generalized Killing $\spin$ spinor on a $3$-dimensional manifold $M$ is equivalent to immerse $M$ in $\mathbb R^4$. But it was proved in \cite{blr13, roth14} that restricting a $\spin$ structure with a spinor field having non-vanishing positive and negative parts is required to get the integrability condition of an immersion in the desired $4$-dimensional target space. Hence, Morel's result has been reformulated for hypersurfaces of $\mathbb R^4$ \cite{LR} because $\mathbb R^4$ has a $\spin$ structure with  positive and negative parallel spinors. The restriction of both spinors to $M$ gives two generalized Killing spinors which, conversely, allow to characterize the immersion of $M$ in $\mathbb R^4$. This result has been extended to other 4-dimensional space forms and product spaces,  that is $\mathbb S^4$, $\mathbb H^4$, $\mathbb S^3 \times \mathbb R$ and $\mathbb H^3\times \mathbb R$ \cite{LR}. In the $\spinc$ setting, the existence of a Codazzi generalized Killing $\spinc$ spinor on a $3$-dimensional manifold $M$ is equivalent to immerse $M$ in the 2-dimensional complex space form  $\mathcal M_2 (k)$ of constant holomorphic sectional curvature $4k$ \cite{NR12}. However here, the condition ``Codazzi" cannot be dropped as in the $\spin$ case, because $\mathcal M_2 (k)$  has only two different $\spinc$ structures (the canonical and the anti-canonical $\spinc$ structures) carrying  each  one parallel spinor lying in the positive half-part of the corresponding $\spinc$ bundles. \\

The aim of the present article is to use $\spinc$ geometry to characterize hypersurfaces of the Riemannian product $\mathbb M_1(c_1) \times \mathbb M_2(c_2)$, where  $\mathbb M_i(c_i)$ denotes the $2$-dimensional space form of constant sectional curvature $c_i \in \mathbb R$. The key starting point is that this  product has two different $\spinc$ structures carrying each a non-vanishing  parallel spinor. The first structure $S_1$ is the product of the canonical $\spinc$ structure on $\mathbb M_1(c_1)$ with the canonical $\spinc$ structure on $\mathbb M_2(c_2)$  and it has a non-vanishing parallel spinor lying in the positive half-part of the $\spinc$ bundle. The second structure $S_2$ is  the product of the canonical $\spinc$ structure on $\mathbb M_1(c_1)$ with the anti-canonical $\spinc$ structure on $\mathbb M_2(c_2)$ and it has a  non-vanishing parallel spinor lying in the negative half-part of the $\spinc$ bundle. Having said that one could expect that restricting both structures $S_1$ and $S_2$, and hence both parallel spinors, to a hypersurface $M$ of $\mathbb M_1(c_1) \times \mathbb M_2(c_2)$ could allow to characterize the immersion. \\

We denote by $\nabla^j$, $\gamma_j$ and $i\Omega^j$  respectively the Clifford multiplication, the $\spinc$ connection and the curvature of the auxiliary line bundle  on the hypersurface $M$ obtained after restricting the $\spinc$ structure $S_j$ on $\mathbb M_{1} (c_1) \times \mathbb M_2(c_2)$ (here $j \in \{1, 2\}$). The main theorem of the paper is:
\begin{thm}\label{thmCM}
Let $\big(M^3,g=(.,.)\big)$ be a simply connected oriented Riemannian manifold endowed with an almost contact metric structure
 $(\Chi,\xi,\eta)$. Let $E$ be a  field of symmetric endomorphisms on $M$, $h$ a function on $M$ and  $V$ a vector field on $M$. Then, the following statements are equivalent:
\begin{enumerate}
\item There exists an isometric immersion of $(M^3,g)$ into $\mathbb M_{1} (c_1) \times \mathbb M_2(c_2)$ with shape operator $E$ and so that, over $M$, the complex structure of $\mathbb M_{1} (c_1) \times \mathbb M_2(c_2)$ is given by $J=\Chi+\eta(\cdot)\nu$, where $\nu$ is the unit normal vector of the immersion and the product structure  is given by $F = f + (V, \cdot) \nu$ for some endomorphism $f$ on $M$.
\item There exists two $\spinc$ structures on $M$ carrying each one a  non-trivial spinor $\varphi_1$ and $\varphi_2$  satisfying
$$\nabla^1_X\varphi_1=-\frac{1}{2}\gamma_1(EX)\varphi_1\ \ \ \text{and}\ \ \  \gamma_1(\xi)\varphi_1= - i\varphi_1.$$
$$\nabla^2_X\varphi_2=\frac{1}{2}\gamma_2(EX)\varphi_2\ \ \ \text{and}\ \ \ \ \gamma_2(V)  \varphi_2 = -i \gamma_2(\xi) \varphi_2 + h \varphi_2.$$
The curvature 2-form $i\Omega^j$ of the connection on  the auxiliary bundle associated with each $\spinc$ structure is given by ($j \in \{1, 2\}$)
$$\left\{
\begin{array}{l}
\Omega^j(e_1, e_2) = \frac 12 (-1)^{j-1}c_1 (h-1) -\frac 12 c_2 (h+1),\\
\Omega^j(e_1, \xi) = \frac 12 \Big( (-1)^{j-1} c_1 - c_2\Big) (e_1, V),\\
\Omega^j (e_2, \xi) = \frac 12 \Big( (-1)^{j-1} c_1 - c_2\Big) (e_2, V),
\end{array}
\right.
$$
in the basis $\{e_1,e_2=\Chi e_1,e_3=\xi\}$.
%\item There exists 2 $\spinc$ structures on $M$ carrying each one a  non-trivial spinor $\varphi_1$ and $\varphi_2$ of constant norms and satisfying
%$$D^1\varphi_1=\frac{3}{2}H\varphi_1\ \ \ \text{and} \ \ \ \gamma_1(\xi)\varphi_1= - i\varphi_1.$$
%$$D^2\varphi_2=-\frac{3}{2}H\varphi_2\ \ \ \text{and} \ \ \ \gamma_2(V)  \varphi_2 = -i \gamma_2(\xi) \varphi_2 + h \varphi_2.$$
%The curvature 2-form of the connection on  the auxiliary bundle associated with these two  $\spinc$ structure are given by ($j \in \{1, 2\}$) 
%$$\Omega^j(e_1, e_2) = \frac 12 (-1)^{j-1}c_1 (h-1) -\frac 12 c_2 (h+1)$$
%$$\Omega^i (e_1, \xi) = \frac 12 \Big( (-1)^{j-1} c_1 - c_2\Big) (e_1, V)$$
%$$\Omega^j (e_2, \xi) = \frac 12 \Big( (-1)^{j-1} c_1 - c_2\Big) (e_2, V)$$
%in the basis $\{e_1,e_2=\Chi e_1,e_3=\xi\}$.
\end{enumerate}
\end{thm}

Again, these two $\spinc$ structures (resp. two generalized Killing $\spinc$ spinors) on $M$ comes from the restriction of the two $\spinc$ structures $S_1$ and $S_2$ (resp. the two parallel spinors) on $\mathbb M_{1} (c_1) \times \mathbb M_2(c_2)$. Needless to say,  when $c_1 = c_2=0$, these two $\spinc$ structures on $M$ coincide and it is in fact the $\spin$ structure coming from the restriction of  the unique $\spin$ structure on $\mathbb R^4$ having positive and negative parallel spinors.  When $c_1\neq 0$ or $c_2 \neq 0$, the two structures in $M$ are different because they are the restriction of the two different structures $S_1$ and $S_2$ .\\

 As an application of Theorem \ref{thmCM},  we prove that totally umbilical hypersurfaces of $\mathbb M_1(c_1) \times \mathbb M_1(c_1)$ and totally umbilical hypersurfaces  of $\mathbb M_1(c_1) \times \mathbb M_2(c_2)$  ($c_1 \neq c_2$)  having a local structure product are of constant mean curvature (see Proposition \ref{pr1} and Proposition \ref{pr2}).

\section{Preliminaries}
In this section we briefly introduce basic facts about $\spinc$ geometry of hypersurfaces (see \cite{LM, montiel, fr1, r1, r2}). Then we describe hypersurfaces of the Riemannian product $\mathbb M_1(c_1) \times \mathbb M_2(c_2)$ \cite{Ko, LTV}, where  $\mathbb M_i(c_i)$ denotes the $2$-dimensional space form of constant sectional curvature $c_i \in \mathbb R$. 

\subsection{Hypersurfaces and induced Spin$^c$ structures}

\underline{{\bf Spin$^c$ structures on manifolds:}} Let $(N^{n+1}, \overline g)$ be a Riemannian $\spinc$ manifold of dimension $n+1\geq 2$ without boundary. On such a manifold, we have  a Hermitian complex vector bundle $\Sigma  N$ endowed with a natural scalar product $(., .)$ and with a connection $\nabla^N $ which parallelizes the metric. This complex vector bundle, called the $\Spinc$ bundle, is endowed with a Clifford multiplication denoted by $``\cdot"$, $\cdot: TN \rightarrow \mathrm{End}_{\mathbb C} (\Sigma N)$, such that at every point $x \in N$, defines an irreducible representation of the corresponding Clifford algebra. Hence, the complex rank of $\Sigma N$ is $2^{[\frac {n+1}{2}]}$. Given a $\Spinc$ structure on $(N^{n+1}, \overline g)$, one can prove that the determinant line bundle $\mathrm{det} (\Sigma N)$ has a root of index $2^{[\frac{n+1}{2}]-1}$. We denote by $L^N$ this root line bundle over $N$ and call it the auxiliary line bundle associated with the $\Spinc$ structure. Locally, a Spin structure always exists. We denote by $\Sigma' N$ the (possibly globally non-existent)
spinor bundle. Moreover, the square root of the auxiliary line bundle $L^N$
always exists locally. But, $\Sigma N = \Sigma' N \otimes (L^N)^{\frac 12}$ exists globally.  This essentially means that, while the spinor bundle and $(L^N)^{\frac 12}$
may not exist globally, their tensor product (the $\Spinc$  bundle) is
defined globally. Thus, the connection $\nabla^N$ on $\Sigma N$ is the twisted connection of the one on the
spinor bundle (coming from the Levi-Civita connection, also denoted by $\nabla^N$) and a fixed connection on $L^N$. \\

We may now define the Dirac operator $D^N$ acting on the space of smooth
sections of $\Sigma N$  by the composition of the metric connection and
the Clifford multiplication. In local coordinates this reads as
\begin{align*}  D^N =\sum_{j=1}^{n+1} e_j \cdot  \nabla^N_{e_j},
 \end{align*}
where $\{e_1,\ldots,e_{n+1}\}$ is a local oriented orthonormal tangent frame. It is a first order elliptic operator, formally self-adjoint with respect to the $L^2$-scalar product and satisfies, for any spinor field $\psi$, the Schr\"odinger-Lichnerowicz formula 
\begin{eqnarray}
(D^N)^2\psi= (\nabla^N)^* \nabla^N\psi+\frac{1}{4}S\psi+\frac{i}{2}\Omega^N\cdot\psi,
%D^2\psi=\frac12 \Delta \vert\psi\vert^2 + \vert\nabla \psi\vert^2 +
%\frac S4 \vert\psi\vert^2 +\frac 12 (i\Omega\cdot\psi, \psi),
\label{bochner}
\end{eqnarray}
where $S$ is the scalar curvature of $N$,  $ (\nabla^N)^*$ is the adjoint of $\nabla^N$ with respect to the $L^2$ scalar product, $i\Omega^N$ is the curvature of the fixed connection on the auxiliary line bundle $L^N$ ($\Omega^N$ is a real $2$-form on $N$) and  $\Omega^N\cdot$ is the extension of the Clifford multiplication to differential forms. For any $X \in \Gamma(TN)$, the Ricci identity is given by
\begin{eqnarray}
\sum_{k=1}^{n+1} e_k \cdot \mathcal{R}^N(e_k,X)\psi =
\frac 12 \Ric^N(X) \cdot \psi -\frac i2 (X\lrcorner\Omega^N)\cdot\psi,
\label{RRicci}
\end{eqnarray}
where $\Ric^N$ is the Ricci curvature of $(N^{n+1}, \overline g)$ and $\mathcal{R}^N$ is the curvature tensor of the spinorial connection $\nabla^N$. In odd dimension, the volume form $\omega_{\C} := i^{[\frac{n+2}{2}]} e_1 \cdot...\cdot e_{n+1}$ acts on $\Sigma N$ as the identity, i.e., $\omega_\C \cdot\psi = \psi$ for any spinor $\psi \in \Gamma(\Sigma N)$. Besides, in even dimension, we have $\omega_\C^2 =1$. We denote by $\Sigma^\pm N$ the eigenbundles corresponding to the eigenvalues $\pm 1$, hence $\Sigma N= \Sigma^+ N\oplus \Sigma^- N$ and a  spinor field $\psi$ can be written $\psi = \psi^+ + \psi^-$. The conjugate $\overline \psi$ of $\psi$ is defined  by $\overline \psi = \psi^+ - \psi^-$.\\

Every $\spin$ manifold has a trivial $\spinc$ structure \cite{fr1}. In fact, we
choose the trivial line bundle with the trivial connection whose curvature 
is zero. Also every K\"ahler manifold $(N, J, \overline g)$ of complex dimension $m$ ($n+1=2m$) has a canonical $\spinc$ structure coming from the complex structure $J$. Let $\ltimes$ be the K\"{a}hler form defined by the complex structure $J$, i.e. $\ltimes (X, Y)= \overline g(JX, Y)$ for all vector fields $X,Y\in \Gamma(TN).$ The complexified tangent bundle $T^\C N =TN \otimes_\R \C$ decomposes into
$$T^\C N = T_{1,0} N\oplus T_{0,1} N,$$
where  $T_{1,0} N$ (resp. $T_{0,1} N$)  is the $i$-eigenbundle (resp. $-i$-eigenbundle) of the complex linear extension of the complex structure. Indeed,
$$T_{1,0}N = \overline{T_{0,1}N} = \{ X- iJX\ \ | X\in \Gamma(TN)\}.$$
Thus, the spinor bundle of the canonical $\spinc$ structure is given by $$\Sigma N = \Lambda^{0,*} N =\bigoplus_{r=0}^m \Lambda^r (T_{0,1}^* N),$$
where $T_{0,1}^*N$ is the dual space of  $T_{0,1} N$. The auxiliary bundle of this canonical $\spinc$ structure 
is given by  $L^N = (K_{N})^{-1}= \Lambda^m (T_{0,1}^* N)$, where $K_{N}= \Lambda^m (T_{1,0}^* N)$ is the canonical bundle of $N$ \cite{fr1}. This line bundle $ L^N$ has a canonical holomorphic connection
 induced from the Levi-Civita connection whose curvature form is given by $i\Omega^N = -i\rho$, where $\rho$ is the Ricci form given by $\rho(X, Y) = \Ric^N(JX, Y)$ for all $X, Y \in \Gamma(TN)$. Hence, this $\spinc$ structure carries  parallel spinors (the constant complex functions) lying in the set of complex functions $\Lambda^{0, 0}N\subset \Lambda^{0, *} N$  \cite{Moro1}.
Of course, we can define another $\spinc$ structure for which the spinor bundle is given by 
$\Lambda^{*, 0} N =\bigoplus_{r=0}^m \Lambda^r (T_{1, 0}^* N)$ and the auxiliary line bundle by $K_{N}$.  This $\spinc$ structure will be called the anti-canonical $\spinc$ structure \cite{fr1} and 
it carries  also parallel spinors (the constant complex functions) lying in the set of complex functions $\Lambda^{0, 0}N\subset \Lambda^{0, *} N$  \cite{Moro1}.\\

For any other $\spinc$ structure on the K\"ahler manifold $N$, the spinorial bundle can be written as \cite{fr1, omu}: $$\Sigma N = \Lambda^{0,*}N \otimes\mathfrak L,$$ where $\mathfrak L^2 = K_{N}\otimes  L^N$ and $L^N$  is the auxiliary bundle associated with this $\spinc$
structure. In this case, the $2$-form $\ltimes$ can be considered as an endomorphism of $\Sigma N$ via
 Clifford multiplication and 
it acts on a spinor field  $\psi$ locally by \cite{kirch, fr1}:
$$\ltimes\cdot\psi =  \frac 12 \sum_{j=1}^{m} e_j\cdot Je_j\cdot\psi.$$
Hence,  we have the well-known orthogonal splitting 
\begin{align}\label{splitt}\Sigma N = \bigoplus_{r=0}^{m}\Sigma_r N,
\end{align}
where $\Sigma_r N$ denotes the eigensubbundle corresponding 
to the eigenvalue $i(m-2r)$ of $\ltimes$, with complex rank $\Big(^m_k\Big)$. The bundle $\Sigma_r N$ corresponds
to $\Lambda^{0, r}N \otimes\mathfrak L$. Moreover,
$$\Sigma^+N = \bigoplus_{r \ \text{even}} \Sigma_r N\ \ \ \text{and} \ \ \ \Sigma^-N = \bigoplus_{r \ \text{odd}} \Sigma_r N.$$
For the canonical (resp. the anti-canonical) $\spinc$ structure, the subbundle $\Sigma_0 N$ (resp. $\Sigma_m N$) is trivial, i.e., $\Sigma_0 N = \Lambda^{0, 0} N \subset \Sigma^+N $ (resp. $\Sigma_m N= \Lambda^{0, 0}N$ which is in $\Sigma^+ N$ if $m$ is even and in $\Sigma^-N$ if $m$ is odd). \\

The product $N_1 \times N_2 $ of two K\"ahler  $\spinc$ manifolds is again a $\spinc$ manifold. We denote by $m_1$ (resp. $m_2$) the complex dimension of $N_1$ (resp. $N_2$). The spinor bundle is identified by $$\Sigma (N_1\times N_2) \simeq \Sigma N_1\otimes \Sigma N_2,$$
via the Clifford multiplication denoted also by ``$\cdot$":
$$(X_1 + X_2)\cdot(\psi_1 \otimes \psi_2) = X_1\cdot\psi_1  \otimes \psi_2 + \overline \psi_1 \otimes X_2\cdot\psi_2, $$
where $X_1 \in \Gamma(TM_1)$, $X_1 \in \Gamma(TM_2)$, $\psi_1 \in \Gamma (\Sigma M_1)$ and $\psi_2 \in \Gamma (\Sigma M_2)$.  We consider the decomposition (\ref{splitt}) of $\Sigma N_1$ and $\Sigma N_2$ with respect to their K\"ahler forms $\ltimes^{N_1}$ and $\ltimes^{N_2}$. Then, the corresponding decomposition of $\Sigma (N_1\times N_2)$ into eigenbundles of $\ltimes^{N_1\times N_2}= \ltimes^{N_1} + \ltimes^{N_2}$ is:
$$\Sigma (N_1\times N_2) \simeq \bigoplus_{k=0}^{m_1+m_2}\Sigma_r(N_1\times N_2),$$
with 
$$\Sigma_r (N_1\times N_2) \simeq \bigoplus_{k=0}^{r} \Sigma_kN_1 \otimes \Sigma_{r-k}N_2,$$
since the K\"ahler form $\ltimes^{N_1\times N_2}$ acts on a section of $\Sigma_kN_1 \otimes \Sigma_{r-k}N_2$ as
$$\ltimes^{N_1\times N_2}(\psi_1\otimes \psi_2) = \ltimes^{N_1}\cdot \psi_1 \otimes\psi_2  +\psi_1\otimes \ltimes^{N_2} \cdot\psi_2 = i (m_1+m_2-2r)\psi_1\otimes\psi_2.$$

%For any other $\spinc$ structure the spinorial bundle can be written as \cite{fr1, omu}:
%$$\Sigma M =  \Lambda^{0, *} M \otimes \mathcal L,$$
% where $\mathcal L^2 = K_M\otimes L$ and  $L$ is the line bundle associated with this $\spinc$ structure. In this case, the 2-form $\Omega$ can be considered as an endomorphism of $\Sigma M$ via Clifford multiplication and it acts on a spinor $\psi$ locally by \cite{fr1, kirch, oussama1}:
%$$\Omega\cdot \psi = \frac{1}{2} \sum_{j=1}^n e_j\cdot Je_j\cdot\psi,$$
%where $\{e_1, e_2,...,e_n\}$ is any local orthonormal frame of tangent vector fields. Moreover, we have the well-known orthogonal splitting
%$$\Sigma M = \oplus_{r=0}^m \Sigma_r M,$$
%where $\Sigma_r M$ denotes the eigensubbundle corresponding to the eigenvalue $i(2r-m)$ of $\Omega$, with complex rank $\Big(^m_k\Big)$. It %is easy to see that $\Sigma_r M \subset \Sigma^+ M$ (resp. $\Sigma_r M \subset \Sigma^-M$) if and only if $r$ is even (resp. $r$ is odd). For %the canonical $\spinc$ structure, we should point out that $\Sigma_0 M$ is trivial \cite{fr1} ($\Sigma_0 M$ is the set of complex functions %on $M$) and so this $\spinc$ structure carries  parallel spinors (the constant functions) lying in $\Sigma_0M$ \cite{Moro1}.\\\\
\underline{{\bf Spin$^c$ hypersurfaces and the Gauss formula:}} Let $N$ be an oriented ($n+1$)-dimensional Riemannian $\spinc$ manifold and $M\subset N$ be an oriented hypersurface. The manifold $M$ inherits a $\spinc$
structure induced from the one on $N$, and we have \cite{r2}
$$ \Sigma M\simeq \left\{
\begin{array}{l}
\Sigma N_{|_M} \ \ \ \ \ \ \text{\ \ \ \ \ \ \ \ \ \ \ \ \ \ \ \ \ if\ $n$ is even,} \\\\
 \Sigma^+ N_{|_M}  \ \text{or}\ \   \Sigma^- N_{|_M}\ \text{\ if\ $n$ is odd.}
\end{array}
\right.
$$
Moreover the Clifford multiplication by a vector field $X$, tangent to $M$, is denoted by $\gamma$ and given by 
\begin{eqnarray}
\gamma(X)\phi = (X\cdot\nu\cdot \psi)_{|_M},
\label{Clifford}
\end{eqnarray}
where $\psi \in  \Gamma(\Sigma N)$ (or $\psi \in \Gamma(\Sigma^+ N)$ if $n$ is odd),
$\phi$ is the restriction of $\psi$ to $M$, ``$\cdot$'' is the Clifford
multiplication on $N$, $\gamma$ that on $M$ and $\nu$ is the unit inner normal
vector. If   $\psi \in \Gamma(\Sigma^- N)$ when $n$ is odd, then we have 
\begin{eqnarray}
\gamma(X)\phi = - (X\cdot\nu\cdot \psi)_{|_M}.
\label{Clifford-}
\end{eqnarray}
The curvature 2-form $i\Omega$ on the auxiliary line bundle $L=L^N_{\vert_M}$ defining the $\Spinc$ structure on $M$ is given by $i\Omega= {i\Omega^N}_{|_M}$. For every
$\psi \in \Gamma(\Sigma N)$ ($\psi \in \Gamma(\Sigma^+ N)$ if $n$ is odd), the real 2-forms
$\Omega$ and $\Omega^N$ are related by \cite{r2}:
\begin{eqnarray}
(\Omega^N \cdot\psi)_{|_M} = \gamma(\Omega)\phi -
\gamma(\nu\lrcorner\Omega^N)\phi.
\label{glucose}
\end{eqnarray}
When $\psi \in \Gamma(\Sigma^-N)$ if $n$ is odd, we have
\begin{eqnarray}
(\Omega^N \cdot\psi)_{|_M} = \gamma(\Omega)\phi +
\gamma(\nu\lrcorner\Omega^N)\phi.
\label{g-}
\end{eqnarray}
We denote by $\nabla$ the spinorial Levi-Civita connection on $\Sigma
M$. For all $X\in \Gamma(TM)$ and $\psi \in \Gamma(\Sigma^+N)$, we have the $\Spinc$ Gauss formula \cite{r2}:
\begin{equation}
(\nabla^{\Sigma N}_X\psi)_{|_M} =  \nabla_X \phi + \frac 12 \gamma(II X)\phi,
\label{spingauss}
\end{equation}
where $II$ denotes the Weingarten map of the hypersurface. If  $\psi \in \Gamma(\Sigma^-N)$, we have 
\begin{equation}
(\nabla^{\Sigma N}_X\psi)_{|_M} =  \nabla_X \phi - \frac 12 \gamma( II X)\phi,
\label{spingauss-}
\end{equation}
for all $X \in \Gamma(TM)$.
%%%%%%%%%%%%%%%%%%%%%%%%%%%%%
%%%%%%%%%%%%%%%

\subsection{Basic facts about $\mathbb M_1(c_1) \times \mathbb M_2(c_2)$ and their real hypersurfaces}
\label{cm}
Let ($\mathbb M_1(c_1) \times \mathbb M_2(c_2), \overline g)$ be the  Riemannian product of $\mathbb M_1(c_1)$ and $ \mathbb M_2(c_2)$, where $M_i(c_i)$ denotes the space form of constant sectional curvature $c_i$ and $\overline g$ denotes the product metric. Consider $\big(M^3, g=(.,.)\big)$  an oriented real hypersurface of $\mathbb M_1(c_1) \times \mathbb M_2(c_2)$  endowed with the metric $g := (\cdot, \cdot)$ induced by $\overline{g}$. The product structure  of $\mathbb P:= \mathbb M_1(c_1) \times \mathbb M_2(c_2)$ is given by the  map $F : T \mathbb P \rightarrow T \mathbb P$  defined, for  $X_1 \in \Gamma(T \mathbb M_1(c_1))$ and
$X_2 \in \Gamma(\mathbb TM_2(c_2))$, by 
\begin{eqnarray}
F(X_1 + X_2) = X_1 - X_2.
\end{eqnarray}
The map $F$ satisfies $F^2 = \mathrm{Id}_{T\mathbb P}, F \neq \mathrm{Id}_{T\mathbb P}$, where $\mathrm{Id}_{T\mathbb P}$ denotes the identity map on $T\mathbb P$. Denoting the Levi-Civita connection on $\mathbb P$ by $\nabla^{\mathbb P}$, we have $\nabla^{\mathbb P} F =0$ and for any $X, Y \in \Gamma(T\mathbb P)$, we have 
$\overline g(FX, Y) = \overline g (X, FY)$.
%The curvature of  $(P, \widetilde g)$ is given by 
%\begin{eqnarray}
%\tilde R (X, Y) Z &=& \frac c4 \big\{ (Y, Z)X - (X, Z) Y + (FY, Z) FX - (FX, Z)FY \\ 
%&& +(Y, Z)FX- (X, Z)FY+ (Y, FZ)X - (X,FZ) Y\big\}
%\end{eqnarray} 
The product structure $F$ induces the existence  on $M$ of a vector $V \in \Gamma(TM)$, a function $h: M \rightarrow \mathbb R$ and an endomorphism $f: TM \rightarrow TM$ such that, for all $X \in \Gamma(TM)$,
\begin{eqnarray}
FX = f X + (V, X) \nu \ \ \ \text{and}\ \ \ \ \ \ F\nu = V +h\nu,
\end{eqnarray}
where $\nu$ is the unit normal vector of the immersion. 
\begin{lem}\label{lem1} 
%\begin{enumerate}%[label=(\rm{\roman{*}}), ref=\rm{\roman{*}}]
The function $f$ is symmetric. Moreover, for any $X \in \Gamma(TM)$, we have 
\begin{eqnarray}f^2X + (V, X)V = X,\label{Structure1} 
\end{eqnarray} 
\begin{eqnarray}\label{Structure2} f V = -hV,
\end{eqnarray}
\begin{eqnarray} \label{Structure3} h^2 + \Vert V \Vert^2 = 1.\end{eqnarray}
%(fe_1, e_1) + (fe_2, e_2) = -2 h
%\begin{eqnarray}
%g(fe_2, fe_2) + (V, e_2)^2 = 1
%\end{eqnarray}
%\begin{eqnarray}
%g(fe_1, fe_1) + (V, e_1)^2 = 1
%\end{eqnarray}
\end{lem}
\begin{proof} First of all, for any $X, Y \in \Gamma(TM)$, we have
\begin{eqnarray*}
(fX, Y) &=& \overline g(fX, Y)= \overline g\big(FX - (V, X)\nu, Y\big) = \overline g(FX, Y) \\
&=& \overline g(X, FY) = \overline g(X, fY + (V, Y)\nu) = (X, fY).
\end{eqnarray*}
Hence $f$ is symmetric. For any $X \in \Gamma(TM)$, $F^2 X = X$. This means that
$$(f+ (V, X) \nu)^2 (X) = X,$$ and hence  
$$\left \{
\begin{array}{l}
f^2 X + (V, X) V = X,\\ 
(V, fX) + h(V, X) = 0,
\end{array}
\right.
$$
which are Equation (\ref{Structure1}) and Equation (\ref{Structure2}). We  also have $F^2 \nu = \nu$. Thus, 
$$V + (V, V) \nu  + h V + h^2 \nu = \nu.$$ This gives $\Vert V \Vert^2 + h^2 = 1$ which is Equation (\ref{Structure3}).
\end{proof}
Moreover, the complex structure $J=J_1+J_2$ on $\mathbb P$ (where $J_i$ denotes the complex structure on $\mathbb M_i(c_i)$) induces on $M$ an almost contact metric structure $\big(\Chi, \xi, \eta, g=(.,.)\big)$, where $\Chi$ is the $(1,1)$-tensor defined, for all $X,Y\in \Gamma(TM)$ by
$$(\Chi X,Y)=\overline{g}(JX,Y).$$
 The tangent vector field $\xi$  and the $1$-form $\eta$ associated with $\xi$ satisfy
 $$\xi=-J\nu\ \ \ \ \text{and}\ \ \ \ \ \ \eta(X)=(\xi,X),$$ 
for all $X\in\Gamma(TM)$. Then, we can easily see  that, for all $X\in\Gamma(TM)$, the following holds:
\begin{eqnarray}
J X = \Chi X + \eta(X) \nu, 
\end{eqnarray}
\beqt
\Chi^2X=-X+\eta(X)\xi,\quad g(\xi,\xi)=1,\quad\text{and}\quad \Chi\xi=0. \label{chi}
\eeqt
Here, we recall that given an almost contact metric structure $(\Chi, \xi, \eta, g)$ one can define a $2$-form 
$\varTheta$ by $\varTheta(X, Y) = g(X, \Chi Y)$ for all $X, Y\in \Gamma(TM)$. Now, $(\Chi, \xi, \eta, g)$ is said to satisfy 
the contact condition if $-2 \varTheta =d\eta$ and if it is the case, $(\Chi, \xi, \eta, g)$ is called a contact metric 
structure on $M$. A contact metric structure $(\Chi, \xi, \eta, g)$  is called a Sasakian structure (and $M$ a Sasaki manifold) if $\xi$ is a Killing vector field (or equivalently, $\Chi = \nabla \xi$) and 
$$(\nabla_X\Chi )Y = \eta (Y) X - g(X, Y) \xi, \ \ \text{for all} \ \ X, Y \in \Gamma(TM).$$
For $\mathbb P$, one can choose a local orthonormal frame $\{e_1, e_2 = \Chi e_1, \xi, \nu\}$ where $\{e_1, e_2 = \Chi e_1, \xi\}$ denotes a  local orthonormal frame of $M$.
\begin{lem}\label{Fo}
We have
%In the basis $\{e_1, e_2 = \Chi e_1, \xi\}$ of $M$, we have 
\begin{enumerate}[label=(\rm{\roman{*}}), ref=\rm{\roman{*}}]
\item\label{as1} $\Chi$ is antisymmetric on $TM$, i.e. $(\Chi e_1, e_2) = -(e_1, \Chi e_2)$ and $\Chi \xi = 0$
\item\label{as2}  $J \circ F = F \circ J$
\item\label{as3} $(V, \Chi X) + \eta(X) h = \eta(fX)$
\item \label{as4}$f\Chi X + \eta (X)V = \Chi fX - (V, X)\xi$
\item \label{as5}$\eta (V) = 0$
\item \label{as6}$f\xi = h\xi - \Chi V$
\item \label{as7}$\eta (fV) = 0$
\item \label{as8}$(fe_1, e_2) = 0$ and  $(fe_1, e_1) = (fe_2, e_2) = -h$
\item \label{as9}$JV = \Chi V$
\item \label{as10}$F\xi = f \xi$
%\item $(fe_1, e_1) + (fe_2, e_2) = -2h$
\end{enumerate}
\end{lem}
\begin{proof}
For any $X, Y \in \Gamma(TM)$, wer have $(\Chi X, \Chi Y) = (X, Y ) - \eta(X)\eta(Y ).$  Thus, 
$(\Chi e_1, e_2) = (\Chi^2 e_1, \Chi e_2) = -(e_1, \Chi e_2)$. It is evident that $\Chi \xi = 0$. This proves (\ref{as1}).  Now, for any $X_1+X_2 \in \Gamma(T\mathbb P)$, we have
\begin{eqnarray*}
J \circ F     (X_1 + X_2) &=& J(X_1 - X_2)= J_1 X_1 - J_2 X_2 \\
&=&F(J_1 X_1 + J_2 X_2)= F \circ J (X_1 + X_2).
\end{eqnarray*}
This proves (\ref{as2}). From $J \circ F = F \circ J$, and using that $f$ is symmetric and (\ref{chi}),  we have for any $X \in \Gamma (TM)$, 
$$\left \{
\begin{array}{l}
f \Chi X + \eta(X) V = \Chi fX - (V, X) \xi,\\ 
(V, \Chi X) + h \eta (X) = \eta (fX).
\end{array}
\right.
$$
This proves (\ref{as3}) and (\ref{as4}). We also have $J (F \nu) = \ F(J \nu)$. Thus, $$\Chi V + \eta (V) \nu - h \xi = -f\xi-(V, \xi) \nu.$$
This implies 
$$\left \{
\begin{array}{l}
f\xi = h \xi - \Chi V, \\
\eta (V) = -(V, \xi) = 0.  
\end{array}
\right.
$$
This proves (\ref{as5}) and (\ref{as6}). From $(V, \Chi X) + h \eta (X) = \eta (fX)$ and for $X = V$, we get 
  $$\eta(fV) = h \eta (V) + {\underbrace{(V, \Chi V) }_{=0}}= 0,$$
  which is (\ref{as7}). We calculate 
   \begin{eqnarray*}
   (fe_1, e_2) &=&  -(f\Chi e_2, e_2) = (-\Chi f e_2 + (V, e_2)\xi, e_2) \\
   &=& -(\Chi fe_2, e_2) = (fe_2, \Chi e_2) = -(fe_2, e_1).
   \end{eqnarray*} 
   Since $f$ symmetric, it implies that $(fe_1, e_2) = 0$. Moreover,  we have 
   \begin{eqnarray*}
   (fe_1, e_1) &=&  -(f\Chi e_2, e_1) = (-\Chi f e_2 + (V, e_2)\xi, e_1) \\
   &=& -(\Chi fe_2, e_1) = (fe_2, \Chi e_1) = (fe_2, e_2).
   \end{eqnarray*} 
   We know that $\mathrm{tr} (F) = 0$. Thus, 
   \begin{eqnarray*}
   0 &=& (Fe_1, e_1) + (Fe_2, e_2) + (F\xi, \xi) + (F\nu, \nu)\\
   &=& (fe_1, e_1) + (fe_2, e_2) + {\underbrace{(f\xi, \xi)}_{= h}}+{\underbrace{(V + h\nu, \nu)}_{h+0=h}}  
  \\ &=& (fe_1, e_1) + (fe_2, e_2). 
   \end{eqnarray*}
Thus, $(fe_1, e_1) = (fe_2, e_2) = -h$. This proves (\ref{as8}).  Since $(V, \xi) = 0$, it is clear that $F \xi = f\xi$ and from $J = \Chi + \eta (\cdot)\nu$, we have $JV =\Chi V$. This proves (\ref{as9}) and (\ref{as10}). 
%\end{enumerate}
%From the relation between the Riemannian connections of $\mathbb S^2 \times \mathbb R^2$ and $M$, $\overline{\nabla}_XY=\nabla_XY+g(IIX,Y)\nu$, we deduce the two following identities: {\color{red}{ces 2 formules sont vrais pour CP2, ils restent vrai pour $\mathbb S^2 \times \mathbb R^2$ ???????}}
%\beqt\label{condd3}
%(\nabla_X\Chi)Y=\eta(Y)IIX-g(IIX,Y)\xi,
%\eeqt
%\beqt\label{condd4}
%\nabla_X\xi=\Chi IIX.
%\eeqt
\end{proof}
For all $X, Y, Z \in \Gamma(TM)$, the Gauss equation for the hypersurface $M$ of $\mathbb P$ can be written as \beq\label{gausssr}
\mathcal R(X,Y)Z&=& \frac {c_1}{4}\Big(( X+fX) \wedge (Y +fY)\Big) +  \frac {c_2}{4}\Big(( X-fX) \wedge( Y -fY)\Big)\nonumber\\
&&+g(IIY,Z)IIX -g(IIX,Z)IIY, 
\eeq
where $\mathcal R$ denotes the Riemann curvature tensor. The Codazzi equation is
\beq\label{codazzisr}
d^{\nabla}II(X,Y) &=& \frac {c_1}{4} \Big(g(fY, Z)g(V, X)-g(fX, Z)g(V, Y) \nonumber \\ && + g(Y, Z)g(V, X) - g(X, Z)g(V, Y)\Big) \nonumber \\ && - \frac{c_2}{4} \Big( g(Y, Z)g(V, X)-g(Y, fZ)g(V, X) \nonumber \\ && - g(X, Z)g(Y, V) + g(X, fZ)g(V, Y)\Big) 
\eeq
Now, we ask if the Gauss equation \eqref{gausssr} and the Codazzi equation \eqref{codazzisr} are sufficient to get an isometric immersion of $(M,g)$ into $\mathbb P= \mathbb M_1(c_1) \times \mathbb M_2(c_2)$.
\begin{defn}[{\bf Compatibility equations}]\label{compp}
Let $(M^3,g)$ be a simply connected oriented Riemannian manifold endowed with an almost contact metric structure $(\Chi, \xi, \eta)$ and $E$ be a field of symmetric endomorphisms on $M$. We say that $(M, g, E)$ satisfies the compatibility equations for $\mathbb M_1(c_1) \times \mathbb M_2(c_2)$  if and only if for any $X,Y,Z\in\Gamma(TM)$, we have
\begin{eqnarray}
\label{gaussSR}%\nonumber
\mathcal R(X,Y)Z&=& \frac {c_1}{4}\Big(( X+fX) \wedge (Y +fY)\Big) +  \frac {c_2}{4}\Big(( X-fX) \wedge( Y -fY)\Big)\nonumber\\
&&+g(EY,Z)EX -g(EX,Z)EY,
\end{eqnarray}
\begin{eqnarray}
\label{codazziSR}
d^{\nabla}E(X,Y) &=& \frac {c_1}{4} \Big(g(fY, Z)g(V, X)-g(fX, Z)g(V, Y) \nonumber \\ && + g(Y, Z)g(V, X) - g(X, Z)g(V, Y)\Big) \nonumber \\ && - \frac{c_2}{4} \Big( g(Y, Z)g(V, X)-g(Y, fZ)g(V, X) \nonumber \\ && - g(X, Z)g(Y, V) + g(X, fZ)g(V, Y)\Big),
\end{eqnarray}
\begin{eqnarray}
\label{Structure4}
(\nabla_Xf)Y = g(Y, V) EX + g(EX, Y)V,
\end{eqnarray} 
\begin{eqnarray}
\label{Structure5}
\nabla_XV = -f (EX) + h EX,
\end{eqnarray}
\begin{eqnarray}\label{Structure6}
\nabla f = -2 EV.
\end{eqnarray}
\end{defn}
In \cite{Ko, LTV}, Kowalczyk and De Lira-Tojeiro-Vit\'orio proved independently that  that the Gauss equation  \eqref{gaussSR} and the Codazzi equation \eqref{codazziSR} together with (\ref{Structure1}), (\ref{Structure2}), (\ref{Structure3}),  (\ref{Structure4}), (\ref{Structure5}), (\ref{Structure6}) and if $\frac{F \pm \mathrm{Id}}{2}$ are of rank $2$, where $F$ is given by $F = 
 \left( \begin{array}{cc}
f &  V\\
V & h
\end{array} \right)
$, are necessary and sufficient for the existence of an isometric immersion from $M$ into $\mathbb M_1(c_1) \times \mathbb M_2(c_2)$ such that the complex structure of $\mathbb M_1(c_1) \times \mathbb M_2(c_2)$ over $M$ is given by $J= \Chi  + \eta(\cdot) \nu$ , $E$ as second fundamental form and such that the product structure coincides with $F$ over $M$. This immersion is global if $M$ is simply connected. Note that this was previously proven in a more abstract way by Piccione and Tausk \cite{PT}.

\section{Isometric immersions into $\mathbb{M}_1(c_1) \times \mathbb{M}_2(c_2)$  via spinors}

In this section, we consider two different $\spinc$ structures on $\mathbb P= \mathbb{M}_1(c_1) \times \mathbb{M}_2(c_2)$  carrying  parallel spinor fields. For the first structure,  the parallel spinor $\psi$ is  lying in $\Sigma^+ \mathbb P$ and for the second  $\spinc$ structure the  parallel spinor field $\Psi$ is lying in $\Sigma^- \mathbb P$.   The restriction of these two  $\spinc$ structures to any hypersurface $M^3$ defines two $\spinc$ structures on $M$, each one  with a generalized Killing spinor field. These spinor fields will characterize the isometric immersion of $M$ into $\mathbb P= \mathbb{M}_1(c_1) \times \mathbb{M}_2(c_2)$.  \\

We denote by $\pi_i(X)$  the projection of a vector $X$ on $T \mathbb M_i(c_i)$. We have
%$$(X_1 + X_2) (\psi_1 \otimes \psi_2) = X_1 \cdot \psi_1  \otimes \psi_2 + \bar \psi_1 \otimes X_2 \cdot \psi_2$$ $$
%and 
\begin{eqnarray}\label{Sys1}
\left \{
\begin{array}{l} 
\pi_1 (V) = \frac{(1-h)V + \Vert V\Vert^2 \nu}{2} \\\\
\pi_2 (V) = \frac{(h+1)V - \Vert V\Vert^2 \nu}{2} \\\\
\pi_1 (\xi) = - \pi_1(J\nu) = -J (\pi_1(\nu)) \\\\
\pi_2 (\xi) = - \pi_2(J\nu) = -J(\pi_2(\nu))\\\\
\pi_1 (\nu) = \frac{(h+1)\nu + V }{2}\\\\
\pi_2(\nu) = \frac{(1-h)\nu -V }{2}
\end{array}
\right.
\end{eqnarray}

\subsection{A first $\spinc$ structure on $\mathbb{M}_1(c_1) \times \mathbb{M}_2(c_2)$  and its restriction to hypersurfaces}
Assume that there exists an isometric immersion of $(M^3,g)$ into $\mathbb{M}_1(c_1) \times \mathbb{M}_2(c_2)$  with shape operator $II$. By Section \ref{cm}, we know that $M$ has an almost contact metric structure $(\Chi,\xi,\eta)$ such that $\Chi X = JX -\eta(X) \nu$ for every $X\in \Gamma(TM)$ and the product structure $F$ on $\mathbb{M}_1(c_1) \times \mathbb{M}_2(c_2)$  will be restricted via $f, V$ and $h$. Consider the product of the canonical $\spinc$ structure on $\mathbb{M}_1(c_1) $  with the canonical one on $ \mathbb{M}_2(c_2)$ . It has a parallel spinor $\psi = \psi_1^+ \otimes \psi_2^+$  lying in $\Sigma_0 (\mathbb{M}_1(c_1) \times \mathbb{M}_2(c_2) )  = \Sigma_0(\mathbb{M}_1(c_1)  ) \otimes \Sigma_0 (\mathbb{M}_2(c_2) ) \subset \Sigma^+ (\mathbb{M}_1(c_1) \times \mathbb{M}_2(c_2))$. 
First of all, using (\ref{splitt}), we have for any $X \in \Gamma(TM)$,
$$J(\pi_2(X))\cdot \pi_2(X)\cdot \psi_2^+ = i \vert \pi_2(X)\vert^2 \psi_2^+ \ \ \text{and}\ \ \  
  J(\pi_1(X))\cdot \pi_1(X)\cdot \psi_1^+ = i \vert \pi_1(X)\vert^2 \psi_1^+.$$ 
  \begin{lem}\label{lemaCM} We have
 $$ - \pi_1(\nu)\cdot \psi_1^+ \otimes \pi_2(\xi) \cdot \psi_2^+  + \pi_1 (\xi)\cdot  \psi_1^+ \otimes \pi_2(\nu)\cdot  \psi_2^+ = 0$$
  \end{lem}
\begin{proof}
 Using that $ i\pi_2(\nu) \cdot \psi_2^+ = J(\pi_2(\nu)) \cdot \psi_2^+$ and $ i\pi_1(\nu)\cdot  \psi_2^+ = J(\pi_1(\nu))\cdot  \psi_2^+$, we have 
 \begin{eqnarray*}
 && - \pi_1(\nu) \cdot \psi_1^+ \otimes \pi_2(\xi)\cdot  \psi_2^+  + \pi_1 (\xi)\cdot  \psi_1^+ \otimes \pi_2(\nu)\cdot  \psi_2^+ \\ &=& i 
 \pi_1(\nu) \cdot \psi_1^+ \otimes \pi_2(\nu)\cdot \psi_2^+ - i \pi_1(\nu) \cdot \psi_1^+ \otimes \pi_2(\nu)\cdot \psi_2^+\\ &=& 0.
 \end{eqnarray*}
 \end{proof}
\begin{lem}
 The restriction $\varphi_1$  of the parallel spinor $\psi$  on $\mathbb{M}_1(c_1) \times \mathbb{M}_2(c_2)$   is a solution of the generalized Killing equation
\begin{eqnarray}
 \nabla^1_X\varphi_1 + \frac 12 \gamma_1 (II X)\varphi_1=0,
\end{eqnarray}
where $\nabla^1$ (resp. $\gamma_1$) denotes the $\Spinc$ Levi-Civita connection (resp. the Clifford multiplication) on the induced $\Spinc$ bundle. Moreover, $\varphi_1$ satisfies $\gamma_1(\xi)\varphi_1= -i\varphi_1$. The curvature 2-form $i\Omega^1$ of the auxiliary line bundle associated with the induced $\spinc$ structure is given in the basis  $\{e_1, e_2 = \Chi e_1, \xi\}$
by 
$$\Omega^1(e_1, e_2) = \frac {c_1}{2} (h-1) - \frac{c_2}{2} (h+1),$$
$$\Omega^1 (e_1, \xi) = \frac {c_1-c_2}{2} (e_1, V),$$
$$\Omega^1 (e_2, \xi) = \frac {c_1-c_2}{2} (e_2, V).$$
%In this case, we have $$
\end{lem}
\begin{proof}
%First, since $\psi$ is parallel, we have $ D^{\SR} \psi = (\nabla^{\SR})^* \nabla^{\SR }\psi =0$. Hence, by the Schr\"{o}dinger-Lichnerowicz formula, we get 
%\begin{eqnarray}
%\Omega^{\SR} \cdot\psi =  i \psi.
%\end{eqnarray}
By the Gauss formula (\ref{spingauss}), the restriction $\varphi_1$ of the parallel spinor $\psi$ on $\mathbb P$ satisfies
$$\nabla^1_X\varphi_1 =- \frac 12 \gamma_1(II )\varphi_1.$$
%Since the spinor $\psi$ is parallel, Equality (\ref{RRicci}) gives
%$$\mathrm{Ric}^{\mathbb S^2 \times \mathbb R^2}(X)\cdot\psi = i (X\lrcorner \Omega^{\SR})\cdot\psi $$
%Where $\Ric$ is the Ricci tensor of $\CM$. Therfore, we compute,
%\begin{eqnarray*}
%\gamma_1(\nu\lrcorner \Omega^{\SR}) \varphi_1 &=&(\nu\lrcorner \Omega^{\SR})\cdot\nu\cdot\psi|_M \\ &=&
%-\nu\cdot(\nu\lrcorner \Omega^{\SR})\cdot\psi|_M\\&=& i  \nu \cdot\mathrm{Ric}^{\SR}\ \nu\cdot\psi_{\vert_M}
%\end{eqnarray*}
%But 
%\begin{eqnarray*}
%\mathrm{Ric}^{\SR} \nu &=& \mathrm{Ric}^{\mathbb S^2} (\pi_1 \nu) = \mathrm{Ric}^{\mathbb S^2} (\frac{\nu+ F \nu}{2}) \\ &=& \frac{\nu + F \nu}{2} = \frac{V + (h+1)\nu}{2}
%\end{eqnarray*}
%Thus, $\gamma_1(\nu\lrcorner \Omega^{\SR})\varphi_1  = - \frac i2 \gamma_1(V)\varphi_1-\frac i2 (h+1) \varphi_1$
%By Equation (\ref{glucose}), we get that 
%\begin{eqnarray}
%\gamma_1(\Omega^1) \varphi_1 = -\frac i2 \gamma_1(V )\varphi_1+ \frac i2 (1-h) \varphi_1
%\label{xi}
%\end{eqnarray}
Now, for any $X, Y \in \Gamma(TM)$, we have 
\begin{eqnarray*}
\Omega^1 (X, Y) &=& \Omega^{\mathbb{M}_1 (c_1) \times \mathbb{M}_2(c_2)} (X, Y) \\&=&
-\mathrm{Ric}^{\mathbb{M}_1 (c_1) } (J\pi_1X, \pi_1Y) - \mathrm{Ric}^{\mathbb{M}_2 (c_2) } (J\pi_2X, \pi_2Y) \\&=& 
-\frac{c_1}{4} g\Big(\Chi X +\eta(X) \nu + \Chi fX +\eta(fX)\nu -(V, X)\xi, Y+fY+(V, Y)\nu \Big) \\&& 
-\frac{c_2}{4} g\Big(\Chi X +\eta(X) \nu - \Chi fX -\eta(fX)\nu +(V, X)\xi, Y - fY-(V, Y)\nu \Big).
\end{eqnarray*} Using Lemma \ref{Fo}, we have 
$$\Omega^1(e_1, e_2) = \frac{c_1}{2} (h-1) - \frac{c_2}{2} (h+1),$$
$$ \Omega^1(e_1, \xi) = \frac{c_1- c_2}{2} (V, e_1),$$ 
$$ \Omega^1(e_2, \xi) = \frac{c_1- c_2}{2} (V, e_2).$$ 
Now, we have 
\begin{eqnarray*}
\gamma_1(\xi) (\varphi_1) &=& {\xi \cdot  \nu  \cdot  (\psi_1^+ \otimes \psi_2^+)}_{\vert_M} \\ &=&
{[\pi_1 (\xi)\cdot  \pi_1(\nu) \cdot \psi_1^+ \otimes \psi_2^+  - \pi_1(\nu) \cdot \psi_1^+ \otimes \pi_2(\xi)\cdot  \psi_2^+]}_{\vert_M} \\ && + {[\pi_1 (\xi) \cdot \psi_1^+ \otimes \pi_2(\nu)\cdot  \psi_2^+ + \psi_1^+ \otimes \pi_2(\xi) \cdot \pi_2(\nu)\cdot  \psi_2^+ ]}_{\vert_M}
\end{eqnarray*}
  Thus, 
  \begin{eqnarray*}
\gamma_1(\xi) (\varphi_1) &=& {\xi \cdot  \nu  \cdot  (\psi_1^+ \otimes \psi_2^+)}_{\vert_M} \\ &=&
[-i \vert \pi_1(\nu) \vert ^2 - i \vert \pi_2 (\nu) \vert^2 ]\varphi_1  - {[\pi_1(\nu) \cdot \psi_1^+ \otimes \pi_2(\xi)\cdot  \psi_2^+  + \pi_1 (\xi)\cdot  \psi_1^+ \otimes \pi_2(\nu) \cdot \psi_2^+]}_{\vert_M}  \\
&=& - i \varphi_1  + {\underbrace{{[- \pi_1(\nu) \cdot \psi_1^+ \otimes \pi_2(\xi) \cdot \psi_2^+  + \pi_1 (\xi)\cdot  \psi_1^+ \otimes \pi_2(\nu)\cdot  \psi_2^+]}_{\vert_M} }_{=0 \ \ \text{by Lemma} \ \ \ref{lemaCM}}} \\&=& -i \varphi_1.
\end{eqnarray*}
\end{proof}
\subsection{A second $\spinc$ structure on $\mathbb{M}_1 (c_1) \times \mathbb{M}_2(c_2)$ and its restriction to hypersurfaces}
One can also endow $\mathbb{M}_1 (c_1) \times \mathbb{M}_2(c_2)$ with another $\spinc$ structure. Mainly, the one coming from the product of the anticanonical $\spinc$ on $\mathbb{M}_1 (c_1)$ with the canonical $\spinc$ structure on $ \mathbb{M}_2(c_2)$ which carries also a parallel spinor $\Psi = \psi_1^- \otimes \psi_2^+ $.  The parallel spinor $\Psi$ lies in $ \Sigma_1 (\mathbb M_1(c_1)) \otimes \Sigma_0( \mathbb M_2(c_2)) \subset \Sigma^-(\mathbb{M}_1 (c_1) \times \mathbb{M}_2(c_2))$.  
Using (\ref{splitt}), we have for any $X \in \Gamma(TM)$ 
  $$J(\pi_2(X))\cdot \pi_2(X)\cdot \psi_2^+ = i \vert \pi_2(X)\vert^2 \psi_2^+ \ \ \text{and}\ \ \  J(\pi_1(X))\cdot \pi_1(X)\cdot \psi_1^- = - \vert \pi_1(X)\vert^2i \psi_1^-.$$
 %  \begin{lem}
   %$$(\pi_1(V) \cdot \pi_1(\nu) \cdot \psi_1^- )\otimes \psi_2^+ = - \psi_1^- \otimes (\pi_2(V)\cdot  \pi_2(\nu)\cdot  \psi_2^+ ) = - \frac 12 \Vert V\Vert^2 (\psi_1^- \otimes \psi_2^+) $$
   %\end{lem}
   %{\bf Proof.} By a straightforward calculation and using (\ref{Sys1}), we get the desired result.    
\begin{lem} \label{lemacl} We have
$$\pi_1(\nu) \cdot \psi_1^-  \otimes (\pi_2(V) + i \pi_2(\xi))\cdot \psi_2^+  - (\pi_1(V) +i \pi_1(\xi)) \cdot \psi_1^- \otimes \pi_2 (\nu)\cdot \psi_2^+ = 0$$
\end{lem}
\begin{proof} Using that $i\pi_2(\nu) \cdot \psi_2^+ = J(\pi_2(\nu))\cdot \psi_2^+$ and  $i\pi_1(\nu)\cdot  \psi_1^- = -J(\pi_1(\nu))\cdot \psi_1^-$, we have 
\begin{eqnarray*}
&& \pi_1(\nu) \cdot\psi_1^-  \otimes (\pi_2(V) + i \pi_2(\xi)) \cdot\psi_2^+  - (\pi_1(V) +i \pi_1(\xi)) \cdot\psi_1^- \otimes \pi_2 (\nu) \cdot\psi_2^+   \\&=&
{\underbrace{ 2 \pi_1(\nu)\cdot\psi_1^- \otimes \pi_2(\nu)\cdot\psi_2^+}_{A}} +  {\underbrace{   \pi_1(\nu)\cdot\psi_1^- \otimes \pi_2(V)\cdot\psi_2^+}_{B}} -{\underbrace{ \pi_1(V)\cdot\psi_1^- \otimes \pi_2(\nu)\cdot\psi_2^+}_{C}}
\end{eqnarray*}
Let's calculate each term of the last identity. First we have 
\begin{eqnarray*}
A &= & 2 \pi_1(\nu)\cdot\psi_1^- \otimes \pi_2(\nu)\cdot\psi_2^+  \\ &=& 
\frac{1-h^2}{2} (\nu \cdot\psi_1^- \otimes \nu \cdot\psi_2^+) - \frac{h+1}{2} (\nu \cdot\psi_1^- \otimes V\cdot \psi_2^+) \\&&
+ \frac{1-h}{2} (V \cdot\psi_1^- \otimes \nu\cdot \psi_2^+) - \frac 12  (V\cdot \psi_1^- \otimes V\cdot \psi_2^+).
\end{eqnarray*}
Next, we have 
\begin{eqnarray*}
B &= &  \pi_1(\nu)\cdot\psi_1^- \otimes \pi_2(V)\cdot\psi_2^+  \\ &=& 
\frac{(h+1)^2}{4} (\nu \cdot\psi_1^- \otimes V\cdot \psi_2^+) - \frac{h+1}{4}\Vert V\Vert^2  (\nu\cdot \psi_1^- \otimes \nu \cdot\psi_2^+) \\&&
+ \frac{h+1}{4} (V\cdot \psi_1^- \otimes V\cdot \psi_2^+) - \frac 14 \Vert V \Vert^2  (V \cdot\psi_1^- \otimes \nu \cdot\psi_2^+),
\end{eqnarray*}
and 
\begin{eqnarray*}
C &= &  \pi_1(V)\cdot\psi_1^- \otimes \pi_2(\nu)\cdot\psi_2^+  \\ &=& 
\frac{(1-h)^2}{4} (V\cdot \psi_1^- \otimes \nu\cdot \psi_2^+) - \frac{1-h}{4} (V \cdot\psi_1^- \otimes V\cdot \psi_2^+) \\&&
+ \frac{1-h}{4} \Vert V \Vert^2 (\nu\cdot \psi_1^- \otimes \nu\cdot \psi_2^+) - \frac 14 \Vert V \Vert^2  (\nu \cdot\psi_1^- \otimes V \cdot\psi_2^+).
\end{eqnarray*}
It's clear that $A+B+C = 0$.
\end{proof}
\begin{lem}
 The restriction $\varphi_2$  of the parallel spinor $\Psi$ (for the  $\spinc$ structure described above) on $\mathbb M_1(c_1) \times \mathbb M_2(c_2)$  is a solution of the generalized Killing equation
\begin{eqnarray}
 \nabla^2_X\varphi_2 =\frac 12 \gamma_2(II X)\varphi_2,
\end{eqnarray}
where $\nabla^2$ (resp. $\gamma_2$) denotes the $\Spinc$ connection (resp. the Clifford multiplication) on the induced $\Spinc$ bundle. Moreover, $\varphi_2$ satisfies $\gamma_2(V)  \varphi_2 = - i \gamma_2(\xi)\varphi_2  + h \varphi_2$. The curvature 2-form of the auxiliary line bundle associated with the induced $\spinc$ structure is given in the basis  $\{e_1, e_2 = \Chi e_1, \xi\}$
by 
$$\Omega^2(e_1, e_2) = -\frac {c_1}{2} (h-1) - \frac{c_2}{2} (h+1),$$
$$\Omega^2 (e_1, \xi) = -\frac {c_1+c_2}{2} (e_1, V),$$
$$\Omega^2 (e_2, \xi) = -\frac {c_1+c_2}{2} (e_2, V).$$
Moreover, we have 
 \begin{eqnarray} \label{Ide1}
  0 =  (\gamma_2 (V) \varphi_2, \varphi_2),
  \end{eqnarray}
  \begin{eqnarray}\label{Ide2}
  g(V, e_1) = -i (\gamma_2 (e_2)\varphi_2, \varphi_2),
  \end{eqnarray}
  \begin{eqnarray}\label{Ide3}
  g(V, e_2) = i (\gamma_2 (e_1)  \varphi_2, \varphi_2),
  \end{eqnarray}
  \begin{eqnarray}\label{Ide4}
  h = i (\gamma_2 (\xi) \varphi_2, \varphi_2).
  \end{eqnarray}

\end{lem}
\begin{proof}
%First, since $\Psi$ is parallel, we have $ D^{\SR} \Psi = (\nabla^{\SR})^* \nabla^{\SR }\Psi =0$. Hence, by the Schr\"{o}dinger-Lichnerowicz formula, we get 
%\begin{eqnarray}
%\Omega^{\SR} \cdot \Psi =  i \Psi.
%\end{eqnarray}
By the Gauss formula (\ref{spingauss}), the restriction $\varphi_2$ of the parallel spinor $\Psi$ on $\mathbb M_{1}(c_1) \times \mathbb M_2(c_2)$ satisfies
$$\nabla^2_X\varphi_2 = \frac 12 \gamma_2(II X)\varphi_2.$$
%Since the spinor $\Psi$ is parallel, Equality (\ref{RRicci}) gives
%$$\mathrm{Ric}^{\SR}(X)\cdot\Psi = i (X\lrcorner \Omega^{\SR})\cdot_2\Psi $$
%Where $\Ric$ is the Ricci tensor of $\SR$. Therfore, we compute,
%\begin{eqnarray*}
%\gamma_2(\nu\lrcorner\Omega_2^{\SR})\varphi_2 &=&-(\nu\lrcorner \Omega^{\SR})\cdot\nu\cdot\Psi|_M \\ &=&
%+\nu\cdot(\nu\lrcorner \Omega_2^{\SR})\cdot\Psi|_M\\&=& -i  \nu \cdot\mathrm{Ric}^{\SR}\ \nu\cdot\Psi_{\vert_M}\\&=& -\frac i2 \gamma_2(V) \varphi_2 + \frac i2 (1 + h) \varphi_2
%\end{eqnarray*}
%By Equation (\ref{glucose}), we get that 
%\begin{eqnarray}
 %\gamma_2(\Omega^2) \varphi_2  =  \frac i2 (1-h)\varphi_2 + \frac i2 \gamma_2(V) \varphi 
%\label{xi2}
%\end{eqnarray}
Now, for any $X, Y \in \Gamma(TM)$, we have 
\begin{eqnarray*}
\Omega (X, Y) &=& \Omega^{\mathbb{M}_1 (c_1) \times \mathbb{M}_2(c_2)} (X, Y) \\&=&
\mathrm{Ric}^{\mathbb{M}_1 (c_1) } (J\pi_1X, \pi_1Y) - \mathrm{Ric}^{\mathbb{M}_2 (c_2) } (J\pi_2X, \pi_2Y) \\&=& 
\frac{c_1}{4} g(\Chi X +\eta(X) \nu + \Chi fX +\eta(fX)\nu -(V, X)\xi, Y+fY+(V, Y)\nu ) \\&& 
-\frac{c_2}{4} g(\Chi X +\eta(X) \nu - \Chi fX -\eta(fX)\nu +(V, X)\xi, Y - fY-(V, Y)\nu).
\end{eqnarray*}
In the basis $\{e_1, e_2 = \Chi e_1, \xi\}$, we have 
$$\Omega(e_1, e_2) = -\frac{c_1}{2} (h-1) - \frac{c_2}{2} (h+1),$$
$$ \Omega(e_1, \xi) = -\frac{c_1+ c_2}{2} (V, e_1),$$ 
$$ \Omega(e_2, \xi) = -\frac{c_1+ c_2}{2} (V, e_2).$$ 
Now, let's calculate
\begin{eqnarray*}
-\gamma_2(\xi) (\varphi_2) &=& {[\xi \cdot  \nu  \cdot  (\psi_1^- \otimes \psi_2^+) ]}_{\vert_M}\\ &=&
{[\pi_1 (\xi) \cdot \pi_1(\nu) \cdot \psi_1^- \otimes \psi_2^+  + \pi_1(\nu) \cdot \psi_1^- \otimes \pi_2(\xi)\cdot  \psi_2^+]}_{\vert_M} \\ && - [{\pi_1 (\xi) \cdot\psi_1^- \otimes \pi_2(\nu) \cdot \psi_2^+ + \psi_1^- \otimes \pi_2(\xi)\cdot  \pi_2(\nu) \cdot\psi_2^+]}_{\vert_M}.
\end{eqnarray*}
Thus, we get 
\begin{eqnarray*}
-\gamma_2(\xi) (\varphi_2) &=& [{\xi \cdot  \nu  \cdot  (\psi_1^- \otimes \psi_2^+)]}_{\vert_M} \\ &=&i (\vert\pi_1(\nu)\vert^2 - \vert \pi_2(\nu)\vert^2) \varphi_2
 + {[\pi_1(\nu)\cdot\psi_1^- \otimes \pi_2(\xi)\cdot \psi_2^+  - \pi_1 (\xi)\cdot  \psi_1^- \otimes \pi_2(\nu)\cdot  \psi_2^+]}_{\vert_M}  \\ &=& 
 ih \varphi_2  + {[\pi_1(\nu) \cdot \psi_1^- \otimes \pi_2(\xi)\cdot  \psi_2^+  - \pi_1 (\xi)\cdot  \psi_1^- \otimes \pi_2(\nu) \cdot \psi_2^+]}_{\vert_M}.
\end{eqnarray*}
In a similar way, we have 
\begin{eqnarray*}
-\gamma_2(V) (\varphi_2) &=&{[ V \cdot  \nu  \cdot  (\psi_1^- \otimes \psi_2^+) ]}_{\vert_M}\\ &=&
{[\pi_1 (V)\cdot \pi_1(\nu)\cdot \psi_1^- \otimes \psi_2^+  + \pi_1(\nu)\cdot \psi_1^- \otimes \pi_2(V) \cdot\psi_2^+ ]}_{\vert_M}\\ && - {[\pi_1 (V) \cdot\psi_1^- \otimes \pi_2(\nu)\cdot \psi_2^+ + \psi_1^- \otimes \pi_2(V)\cdot \pi_2(\nu)\cdot \psi_2^+  ]}_{\vert_M}\\ &=& 
{[\pi_1(\nu)\cdot \psi_1^- \otimes \pi_2(V)\cdot \psi_2^+  - \pi_1 (V)\cdot \psi_1^- \otimes \pi_2(\nu)\cdot \psi_2^+]}_{\vert_M}.
\end{eqnarray*}
Now, we have 
\begin{eqnarray*}
&& [-\gamma_2(V) - i \gamma_2(\xi) ] (\varphi_2) \\ &=& {[\pi_1(\nu)\cdot \psi_1^- \otimes \pi_2(V)\cdot \psi_2^+  - \pi_1 (V) \cdot\psi_1^- \otimes \pi_2(\nu) \cdot\psi_2^+ ]}_{\vert_M}\\&&  -h \varphi_2  +{[ i \pi_1(\nu) \cdot\psi_1^- \otimes \pi_2(\xi)\cdot \psi_2^+  - i \pi_1 (\xi) \cdot\psi_1^- \otimes \pi_2(\nu)\cdot \psi_2^+]}_{\vert_M} \\ 
&=&  -h \varphi_2  \\ && {\underbrace{+{[\pi_1(\nu)\cdot \psi_1^-  \otimes (\pi_2(V) + i \pi_2(\xi))\cdot \psi_2^+  - (\pi_1(V) +i \pi_1(\xi)) \cdot\psi_1^- \otimes \pi_2 (\nu)\cdot \psi_2^+]}_{\vert_M}}_{= 0 \ \ \text{by Lemma}\ \ \ref{lemacl}}} \\ &=& - h \varphi_2.
\end{eqnarray*}
%\begin{eqnarray} \label{35}
%\frac i2 (1-h) \varphi_2 + \frac i2 \gamma_2(V) \varphi_2 = - \frac 12 (h-1) \gamma_2(\xi) \varphi_2 - \frac 12 \gamma_2(V) \gamma_2 (\xi)  \varphi_2
%\end{eqnarray}
%Taking the Clifford multiplication of the last identity by $V \neq 0$, we get  
  %\begin{eqnarray}\label{36}
%\frac i2 (1-h) \gamma_2 (V) \varphi_2 -\frac i2 \Vert V\Vert^2  \varphi_2 = - \frac 12 (h-1) \gamma_2 (V)\gamma_2 (\xi)\varphi_2 + \frac 12 \Vert V \Vert^2 \gamma_2 (\xi) \varphi_2
%\end{eqnarray}
%Inserting (\ref{35}) into (\ref{36}) and using that $h^2 + \Vert V \Vert^2 = 1$, we ge
   Taking the scalar product of the last identity with $\varphi_2$, then the real part of the scalar product  with $\gamma_2 (e_1) \varphi_2$, then with $\gamma_2 (e_2) \varphi_2$,  we get  (\ref{Ide1}), (\ref{Ide2}), (\ref{Ide3}) and (\ref{Ide4}). 
\end{proof}

\section{Generalized Killing $\spinc$ spinors and isometric immersions}

\begin{lem} \cite{NR12}
Let $E$ be a field of symmetric endomorphisms on a $\spinc$ manifold $M^3$ of dimension $3$, then 
\begin{eqnarray}
 \gamma(E(e_i))\gamma( E(e_j)) -  \gamma(E(e_j)) \gamma( E(e_i ))&=& 2 (a_{j3} a_{i2} - a_{j2}a_{i3}) e_1 \nonumber \\&& +2(a_{i3}a_{j1} - a_{i1}a_{j3}) e_2 \nonumber \\ && + 2(a_{i1}a_{j2} - a_{i2}a_{j1})e_3,
\end{eqnarray}
where $(a_{ij})_{i,j}$ is the matrix of $E$ written in any local orthonormal frame of $TM$.
\label{aij}
\end{lem}
\begin{prop}
Let $(M^3,g)$ be a Riemannian $\spinc$ manifold endowed with an almost contact metric structure $(\Chi,\xi,\eta)$. Assume that there exists a vector $V$ and a function $h$ and  a $\spinc$ structure with  non-trivial spinor $\varphi_1$  satisfying
$$\nabla^1_X\varphi_1=-\frac{1}{2}\gamma_1(EX)\varphi_1 \ \ \ \text{and}\ \ \ \ \gamma_1(\xi)\varphi_1=-i\varphi_1,$$
where $E$ is a field of symmetric endomorphisms on $M$. Moreover, we suppose that the curvature 2-form of the connection 
on the auxiliary line bundle associated with the $\spinc$ structure is given by
$$\Omega^1(e_1, e_2) = \frac{c_1}{2} (h-1) - \frac{c_2}{2} (h+1),$$
$$ \Omega^1(e_1, \xi) = \frac{c_1- c_2}{2} (V, e_1),$$ 
$$ \Omega^1(e_2, \xi) = \frac{c_1- c_2}{2} (V, e_2),$$ 
 in the basis $\{e_1,e_2=\Chi e_1,e_3=\xi\}$.
Hence, the Gauss equation is satisfied for $\mathbb M_{1} (c_1) \times \mathbb M_2(c_2)$ if and only if the  Codazzi equation for $\mathbb M_{1} (c_1) \times \mathbb M_2(c_2)$ is satisfied.
\label{GGCC}
\end{prop}
\begin{proof}
We compute the spinorial curvature $\mathcal{R}^1$ on $\varphi_1$, we get
$$\mathcal{R}^1_{X, Y} \varphi_1= -\frac 12 \gamma_1(d{^\nabla} E(X, Y))\varphi_1 + \frac 14 \big(\gamma_1(EY)\gamma_1(EX) - \gamma_1(EX)\gamma_1( EY)\big)\varphi_1.$$
In the basis $\{e_1, e_2 = \Chi e_1, e_3 = \xi\}$, the Ricci identity (\ref{RRicci}) gives that 
\begin{eqnarray*}
 \frac 12 \gamma_1(\mathrm{Ric}(X))\varphi_1 -\frac i2 \gamma_1(X\lrcorner\Omega^1)\varphi_1 &=& \frac 14 \sum_{k=1}^3 \gamma_1(e_k) \big(\gamma_1(EX)\gamma_1( Ee_k) -\gamma_1( Ee_k)\gamma_1(EX)\big)\varphi_1 \\&&
-\frac 12 \sum_{k=1}^3 \gamma_1(e_k)\gamma_1( d^{\nabla} E(e_k, X ))\varphi_1.
\end{eqnarray*}
By Lemma \ref{aij} and for $X= e_1$, the last identity  becomes   
\begin{eqnarray}
 &\ &(\mathrm{R}_{1221} +\mathrm{R}_{1331} -a_{11}a_{33} - a_{11}a_{22} + a_{13}^2 +a_{12}^2 + \frac {c_1}{2}(h-1)-\frac{c_2}{2}(h+1) )\gamma_1(e_1 )\varphi_1 \nonumber\\ && +( \mathrm{R}_{1332} -a_{12}a_{33} + a_{32}a_{13})\gamma_1(e_2) \varphi_1 \nonumber \\ && + (\mathrm{R}_{1223} -a_{22}a_{13} + a_{32}a_{12})\gamma_1(e_3) \varphi_1  \nonumber\\
 &&-\frac {c_1-c_2}{2} (V, e_1) \varphi_1
   \\ &=&
- \gamma_1(e_2)\gamma_1( d^{\nabla}E (e_2, e_1))\varphi_1- \gamma_1(e_3)\gamma_1( d^{\nabla}E (e_3, e_1))\varphi_1.\nonumber
\label {e1varphi}
\end{eqnarray}
Since $\vert \varphi\vert$ is constant ($\vert \varphi \vert =1$), the set $\{\varphi_1, \gamma_1(e_1)\varphi_1, \gamma_1(e_2)\varphi_1, \gamma_1(e_3)\varphi_1\}$ is an orthonormal frame of $\Sigma M$ with respect to the real scalar product $\pre ( ., .)$. Hence, from Equation (\ref{e1varphi}) we deduce
\begin{eqnarray*}
\mathrm{R}_{1221} +\mathrm{R}_{1331} -(a_{11}a_{33} + a_{11}a_{22} - a_{13}^2 -a_{12}^2 ) &+& \frac {c_2}{2} (h-1) -\frac{c_2}{2} (h+1) \\ &=& g(d^\nabla E(e_1, e_2), e_3) - g(d^\nabla E(e_1, e_3), e_2) 
 \end{eqnarray*}
 \begin{eqnarray*}
\mathrm{R}_{1332} -(a_{12}a_{33} - a_{32}a_{13}) &=& g(d^\nabla E(e_1, e_3), e_1)\\
\mathrm{R}_{1223} -(a_{22}a_{13} - a_{32}a_{12})&=& -g(d^\nabla E(e_1, e_2), e_1)\\
 - \frac {(c_1-c_2)}{2} (V, e_1) &= & g(d^\nabla E(e_2, e_1), e_2) + g(d^\nabla E(e_3, e_1), e_3) 
\end{eqnarray*}
The same computation holds for the unit vector fields $e_2$ and $e_3$ and we get
\begin{eqnarray*}
\mathrm{R}_{2331}  -(a_{12}a_{33} - a_{13}a_{23})& =& - g(d^\nabla E(e_2, e_3), e_2)\\
\mathrm{R}_{2332} + \mathrm{R}_{2112}  -(a_{22}a_{33} + a_{22}a_{11} - a_{13}^2 - a_{12}^2) && + \frac {c_1}{2}(h-1) - \frac{c_2}{2} (h+1)  \\ &=& g(d^\nabla E(e_2, e_3), e_1) +g(d^\nabla E(e_1, e_2), e_3)\\
\mathrm{R}_{2113} -(a_{23}a_{11} - a_{12}a_{13}) &=& - g(d^\nabla E(e_1, e_2), e_2)\\
 -\frac {(c_1-c_2)}{2} (V, e_2)&=& g(d^\nabla E(e_1, e_2), e_1) +  g(d^\nabla E(e_3, e_2), e_3)\\
\mathrm{R}_{3221}  -(a_{13}a_{22} - a_{23}a_{21}) -\frac {(c_1-c_2)}{2} (V, e_2)  &=& - g(d^\nabla E(e_2, e_3), e_3)\\
\mathrm{R}_{3112}  - (a_{32}a_{11} - a_{31}a_{12}) + \frac {(c_1-c_2)}{2} (V, e_1) &=& g(d^\nabla E(e_1, e_3), e_3)\\
\mathrm{R}_{3113}+ \mathrm{R}_{3223}-(a_{22}a_{33} - a_{11}a_{33} + a_{13}^2 + a_{23}^2) &= & g(d^\nabla E(e_2, e_3), e_1) - g(d^\nabla E(e_1, e_3), e_2)\\
g(d^\nabla E(e_2, e_3), e_2) &=& - g(d^\nabla E(e_1, e_3), e_1)
\end{eqnarray*}
The last twelve equations will be called System 1 and it is clear that  the Gauss equation for $\mathbb M_1(c_1) \times \mathbb M_2(c_2)$ is satisfied if and only if the Codazzi equation
 for $\mathbb M_1(c_1) \times \mathbb M_2(c_2)$ is satisfied.
 \end{proof}
\begin{lem} Under the same condition as Proposition \ref{GGCC}, we have $\nabla_X \xi = \Chi EX$.
%$$(\nabla_X\Chi) Y = \eta (Y) EX- g(EX, Y)\xi$$
\label{chii}
\end{lem}
\begin{proof}
In fact, we simply compute the derivative of $\gamma_1(\xi)\varphi_1= -i \varphi_1$ in the direction of $X\in \Gamma (TM)$ to get
 
 \begin{eqnarray*}
 \gamma_1(\nabla_X \xi)\varphi &=& \frac i2 \gamma_1(EX)\varphi _1+ \frac 12 \gamma_1(\xi)\gamma_1( EX)\varphi_1 
\end{eqnarray*}
Using that $-i \gamma_1(e_2)\varphi_1 = \gamma_1(e_1)\varphi_1$, the last equation reduces to 
$$\gamma_1(\nabla_X\xi)\varphi_1 - g(EX, e_1) \gamma_1(e_2)\varphi_1 + g(EX, e_2) \gamma_1(e_1)\varphi_1 =0.$$
Finally $\nabla_X \xi= \Chi EX$. 
\end{proof}
%Now, we compute the derivative of $-i e_2\bullet\phi = e_1\bullet\phi$ in the direction of $e_1$ to get
%$$\nabla_{e_1} (\Chi e_1)\bullet\phi -\frac 12 e_2\bullet Ee_1\bullet \phi = i \nabla_{e_1}e_1 \bullet \phi -\frac i2 e_1\bullet Ee_1\bullet\phi.$$
%But, using that $\xi\bullet\phi = -i\phi$, we have
%$$\frac 12 e_2\bullet Ee_1\bullet \phi -\frac i2 e_1\bullet Ee_1\bullet\phi = -a_{11} \xi\bullet\phi -a_{12} \phi.$$
%Denoting by $\Gamma_{ij}^k$ the Christoffel symbols of $\{e_1, \Chi e_1, \xi\}$, we have $\nabla_{e_1}e_1 = \Gamma_{11}^1 e_1 + \Gamma_{11}^2 e_2 + \Gamma_{11}^3 e_3$. Moreover, using that $\nabla_{e_1}e_3 = \Chi Ee_1$, we get
%$$\Gamma_{11}^3 = g(\nabla_{e_1}e_1, e_3) = -g (e_1, \nabla_{e_1}e_3) = a_{12}.$$
%Hence, $\nabla_{e_1} (\Chi e_1)\bullet\phi = -a_{11} \xi\bullet\phi +  \Gamma_{11}^1 e_2\bullet\phi +  \Gamma_{11}^2 e_2\bullet\phi.$ Finally
%$$\nabla_{e_1} (\Chi e_1)\bullet\phi - \Chi(\nabla_{e_1}e_1)\bullet\phi = -a_{11} \xi\bullet\phi,$$
%which is Equation (\ref{cond3}) for $X = Y =e_1$. Similary, 
%we compute the derivative of $-i e_2\bullet\phi = e_1\bullet\phi$ in the direction of $e_2$ and $\xi$ to get Equation 
%(\ref{cond3}) for any $X, Y\in \Gamma(TM)$. 
\begin{prop}
Let $(M^3,g)$ be a Riemannian $\spinc$ manifold endowed with an almost contact metric structure $(\Chi,\xi,\eta)$. Assume that there exist a nonzero vector field  $V$ and a function $h$ such that there exists a $\spinc$ structure with  non-trivial spinor $\varphi$  satisfying
$$\nabla^2_X\varphi=\frac{1}{2}\gamma_2(EX)\varphi \ \ \ \text{and}\ \ \ \ \ \gamma_2(V)\varphi_2 = -i \gamma_2(\xi)\varphi_2 + h \varphi_2,$$
where $E$ is a field of symmetric endomorphisms on $M$. Moreover, we suppose that the curvature 2-form of the connection 
on the auxiliary line bundle associated with the $\spinc$ structure is given by
$$\Omega^2(e_1, e_2) = -\frac {c_1}{2} (h-1) - \frac{c_2}{2} (h+1)$$
$$\Omega^2 (e_1, \xi) =- \frac {(c_1+c_2)}{2}(e_1, V)$$
$$\Omega^2 (e_2, \xi) = -\frac {(c_1+c_2)}{2} (e_2, V)$$
 in the basis $\{e_1,e_2=\Chi e_1,e_3=\xi\}$.
The Gauss equation for $\mathbb M_1(c_1) \times \mathbb M_2(c_2)$ is satisfied if and only if the Codazzi equation
 for $\mathbb M_1(c_1) \times \mathbb M_2(c_2)$ is satisfied.
\label{GGCC2}
\end{prop}
\begin{proof} First, from $ \gamma_2(V)\varphi_2 = -i \gamma_2(\xi)\varphi_2 + h \varphi_2$,  we have that 
(\ref{Ide1}), (\ref{Ide2}), (\ref{Ide3}) and (\ref{Ide4}) are satisfied. 
We compute the spinorial curvature $\mathcal{R}^2$ on $\varphi_2$, we get
$$\mathcal{R}^2_{X, Y} \varphi_2= \frac 12 \gamma_2(d{^\nabla} E(X, Y))\varphi_2 + \frac 14 \big(\gamma_2(EY)\gamma_2( EX) - \gamma_2(EX) \gamma_2(EY)\Big)\varphi_2.$$
In the basis $\{e_1, e_2 = \Chi e_1, e_3 = \xi\}$, the Ricci identity (\ref{RRicci}) gives that 
\begin{eqnarray*}
 \frac 12 \gamma_2(\mathrm{Ric}(X))\varphi_2 -\frac i2 \gamma_2(X\lrcorner\Omega^2)\varphi_2 &=& \frac 14 \sum_{k=1}^3 \gamma_2(e_k)\big( \gamma_2(EX)\gamma_2( Ee_k) - \gamma_2(Ee_k)\gamma_2( EX)\big)\varphi_2\\&+&
\frac 12 \sum_{k=1}^3  \gamma_2(e_k) \gamma_2(d^{\nabla} E(e_k, X ))\varphi_2.
\end{eqnarray*}
By Lemma \ref{aij} and for $X= e_1$, the last identity  becomes   
\begin{eqnarray}
 &\ &(\mathrm{R}_{1221} +\mathrm{R}_{1331} -a_{11}a_{33} - a_{11}a_{22} + a_{13}^2 +a_{12}^2) \gamma_2(e_1) \varphi_2 \\ && + \frac i2 (c_1+c_2)(V, e_1) \gamma_2(\xi)\varphi_2  + \frac i2 [c_1(h-1)+c_2(h+1)] \gamma_2(e_2) \varphi_2  \nonumber\\ && +( \mathrm{R}_{1332} -a_{12}a_{33} + a_{32}a_{13})\gamma_2(e_2) \varphi_2 \nonumber \\ && + (\mathrm{R}_{1223} -a_{22}a_{13} + a_{32}a_{12})\gamma_2(e_3)\varphi_2  \nonumber\\
 &=&
\gamma_2( e_2) \gamma_2(d^{\nabla}E (e_2, e_1))\varphi_2 + \gamma_2(e_3)\gamma_2(d^{\nabla}E (e_3, e_1))\varphi_2. \nonumber
\label {e1varphi}
\end{eqnarray}
Since $\vert \varphi\vert$ is constant ($\vert \varphi \vert =1$), the set $\{\varphi_2, \gamma_2(e_1)\varphi_2, \gamma_2(e_2) \varphi_2, \gamma_2(e_3)\varphi_2\}$ is an orthonormal frame of $\Sigma M$ with respect to the real scalar product $\pre ( ., .)$. Hence, from Equation (\ref{e1varphi}) we deduce
\begin{eqnarray*}
  && \mathrm{R}_{1221} +\mathrm{R}_{1331} -(a_{11}a_{33} + a_{11}a_{22} - a_{13}^2 -a_{12}^2 ) \\ & -&  
   \frac 12 (c_1 +c_2) (V, e_1)^2   - \frac 12 h\Big( c_1(h-1)+c_2(h+1)\Big)   \\ &=&  g(d^\nabla E(e_2, e_1), e_3) - g(d^\nabla E(e_3, e_1), e_2)
   \end{eqnarray*}
   \begin{eqnarray*}
\mathrm{R}_{1332} -(a_{12}a_{33} - a_{32}a_{13}) - \frac 12 (c_1 +c_2)(V, e_1)(V, e_2)  &=& - g(d^\nabla E(e_1, e_3), e_1)\\
\mathrm{R}_{1223} -(a_{22}a_{13} - a_{32}a_{12}) + \frac 12 \Big(c_1(h-1)+c_2(h+1)\Big) (V, e_2)&=& -g(d^\nabla E(e_2, e_1), e_1)\\
 \frac 1 2  (c_1 +c_2)h(V, e_1)- \frac 12 \Big( c_1(h-1)+c_2(h+1)\Big) (V, e_1)  &= & -g(d^\nabla E(e_2, e_1), e_2) \\ &&  - g(d^\nabla E(e_3, e_1), e_3) 
\end{eqnarray*}
The same computation holds for the unit vector fields $e_2$ and $e_3$ and we get
\begin{eqnarray*}
\mathrm{R}_{2331}  -(a_{12}a_{33} - a_{13}a_{23}) -\frac 12 (c_1 +c_2)(V, e_2) (V, e_1) & =&  g(d^\nabla E(e_2, e_3), e_2)
\end{eqnarray*}
\begin{eqnarray*}
&& \mathrm{R}_{2332} + \mathrm{R}_{2112}  -(a_{22}a_{33} + a_{22}a_{11} - a_{13}^2 - a_{12}^2)      \\ & -& \frac 12(c_1+c_2) (e_2, V)^2 -  \frac 12 h\Big(c_1(h-1)+c_2(h+1)\Big) 
\\
&=& -g(d^\nabla E(e_2, e_3), e_1) -g(d^\nabla E(e_1, e_2), e_3)
\end{eqnarray*}
\begin{eqnarray*}
\mathrm{R}_{2113} -(a_{23}a_{11} - a_{12}a_{13}) -\frac 12 \Big(c_1(h-1)+c_2(h+1)\Big) (V, e_1) &=&  g(d^\nabla E(e_1, e_2), e_2)\\
  -\frac 12 \Big( c_1(h-1) +c_2(h+1)\Big) (V, e_2)  + \frac 12(c_1 +c_2) h (V, e_2) &=& - g(d^\nabla E(e_1, e_2), e_1) \\ &&  -  g(d^\nabla E(e_3, e_2), e_3)\\
\mathrm{R}_{3221}  -(a_{13}a_{22} - a_{23}a_{21}) +\frac 12 (c_1 +c_2)h(V, e_2)   &=&  g(d^\nabla E(e_2, e_3), e_3)\\
\mathrm{R}_{3112}  - (a_{32}a_{11} - a_{31}a_{12}) - \frac 12 (c_1 +c_2)h (V, e_1) &=& -g(d^\nabla E(e_1, e_3), e_3)\end{eqnarray*}
\begin{eqnarray*}
&& \mathrm{R}_{3113}+ \mathrm{R}_{3223}-(a_{22}a_{33} - a_{11}a_{33} + a_{13}^2 + a_{23}^2) \\   &-& \frac 12 (c_1 +c_2)(e_1, V)^2  - \frac 12 (c_1 +c_2)(V, e_2)^2 \\ &= & -g(d^\nabla E(e_2, e_3), e_1) +g(d^\nabla E(e_1, e_3), e_2)\end{eqnarray*}
\begin{eqnarray*}
   {\underbrace{-\frac i2 (e_1, V) (\gamma_2(e_1) \varphi_2, \varphi_2) - \frac i2 (V, e_2) (\gamma_2(e_2) \varphi_2, \varphi_2)}_{(\gamma_2(V)\varphi_2, \varphi_2)= 0}} &=& - g(d^\nabla E(e_2, e_3), e_2)  - g(d^\nabla E(e_1, e_3), e_1)
\end{eqnarray*}
The last twelve equations will be called System 2 and it is clear that   the Gauss equation for $\mathbb M_{1} (c_1) \times \mathbb M_2(c_2)$ is satisfied if and only if the Codazzi equation
 for $\mathbb M_{1} (c_1) \times \mathbb M_2(c_2)$ is satisfied. \end{proof}

\subsection{Spinorial characterization of hypersurfaces of $\mathbb M_{1} (c_1) \times \mathbb M_2(c_2)$}
The main goal of this section is to prove Theorem \ref{thmCM}.
\begin{proof}[Proof of Theorem \ref{thmCM}]
 It is clear from the previous section that Assesrtion 1 implies Assertion 2. Now, Assume that Assertion 2 holds.
 We have to establish the compatibility equation (\ref{gaussSR}), (\ref{codazziSR}), (\ref{Structure1}), (\ref{Structure2}), (\ref{Structure3}), (\ref{Structure4}), (\ref{Structure5}) and (\ref{Structure6}). 
 First we define $f: TM \rightarrow TM$ by $(fe_1, e_1) =(fe_2, e_2)= -h, (fe_1, e_2) = 0$ and $(f\xi, e_1) = (V, e_2)$, $(f\xi, e_2) = -(V, e_1)$. 
Since $\gamma_2(V)  \varphi_2 = -i \gamma_2(\xi) \varphi_2 + h \varphi_2$. It is clear that 
$$\left\{
\begin{array}{l}
h^2 + \Vert V \Vert^2 = 1\\
(V, e_1) = -i (\gamma_2(e_2)\varphi_2, \varphi_2)\\
(V, e_2) = i (\gamma_2(e_1)\varphi_2, \varphi_2)\\
(V, \xi) = 0\\
fV = -hV\\
f^2 = \mathrm{Id} - (V, \cdot) V
\end{array}
\right.
$$
So Equations  (\ref{Structure1}), (\ref{Structure2}), (\ref{Structure3}) are satisfied. Moreover, by Lemma \ref{chii}, we have
\begin{eqnarray*}
(\nabla_{e_1} V , \xi) &=& (V, e_1) (\nabla_{e_1} e_1, \xi) + (V, e_2) (\nabla_{e_1}e_2, \xi) \\ &=&
-(V, e_1) (e_1 \Chi Ee_1) - (V, e_2) (e_2, \Chi Ee_1) \\ &=&
(V, e_1) E_{12} - (V, e_2) E_{11} + E_{11} (V, e_1) - E_{12} (V, e_2).
\end{eqnarray*}
By a similar computation, we get that Equation (\ref{Structure5}) is satisfied. Now, Equation (\ref{Structure6}) is also satisfied because 
\begin{eqnarray*}
X(h) &=& (i \gamma_2({\underbrace{\nabla_X\xi}_{=\Chi EX}})\varphi_2, \varphi_2) + \frac i2 (\gamma_2(\xi) \gamma_2(EX)\varphi_2, \varphi_2) - \frac i2 (\gamma_2(EX) \gamma_2(\xi )\varphi_2, \varphi_2) \\ &=&
i E(X, e_1) (\gamma_2(e_2)\varphi_2, \varphi_2) - i E(X, e_2) (\gamma_2(e_1 )\varphi_2, \varphi_2)
\\ && \frac i2 E(X, e_2) (\gamma_2(\xi)\gamma_2(e_1)\varphi_2, \varphi_2) - \frac i2 E(X, e_1) (\gamma_2(e_1 )\gamma_2( \xi) \varphi_2, \varphi_2) \\
&& +\frac i2 E(X, e_2) (\gamma_2(\xi)\gamma_2( e_2)\varphi_2, \varphi_2) - \frac i2 E(X, e_2) (\gamma_2(e_2)  \gamma_2(\xi) \varphi_2, \varphi_2)
\\ && - \frac i2 E(X, \xi) \vert \varphi_2\vert^2 + \frac i2 E(X, \xi) \vert \varphi_2\vert^2 \\ 
&=& -(V, e_1) E(X, e_1) - E(X, e_2) (V, e_2) \\
&& \frac i2 E(X, e_1) (\gamma_2(e_2 ) \varphi_2, \varphi_2) + \frac i2 E(X, e_1) (\gamma_2(e_2)\varphi_2, \varphi_2) \\
&& - \frac i2 E(X, e_2) (\gamma_2(e_1) \varphi_2, \varphi_2) - \frac i2 E(X, e_2) (\gamma_2(e_1) \varphi_2, \varphi_2) \\
&=& - 2 E(X, e_2) (V, e_2) - 2 E(X, e_1) (V, e_1) \\
&=& -2 (EV, X).
\end{eqnarray*}
Now, we have 
\begin{eqnarray*}
(\nabla_{e_1} fe_1, e_1 ) &=& e_1 (fe_1, e_1) + (fe_1, \xi) (\nabla_{e_1}\xi, e_1) \\
&=& e_1 (-h) + (V, e_2) (\Chi Ee_1, e_1) \\
&=& 2 (EV, e_1)- (V, e_2) E_{12} \\
&=& 
2 E_{11} (e_1, V) - (e_2, V) E_{12},
\end{eqnarray*}
and
\begin{eqnarray*}
(f(\nabla_{e_1} e_1) , e_1) &=& (\nabla_{e_1} e_1, fe_1) \\ 
&=& (fe_1, e_2) (\nabla_{e_1}{e_1}, e_2 ) + (fe_1, \xi) (\nabla_{e_1}e_1, \xi) \\ 
&=& -(\nabla_{e_1}{\xi}, e_1) (V, e_2) = -(\Chi Ee_1, e_1) (V, e_2) = E_{12} (V, e_2)
\end{eqnarray*}
Thus, $((\nabla_{e_1}f) e_1, e_1 )  = 2 E_{11} (e_1, V) $. By a similar computation, one can get (\ref{Structure4}). Solving System 1 and System 2 simultaneously gives the Gauss and the Codazzi equations. Finally, we have to check that  $\frac{F +\mathrm{Id}}{2}$
and $\frac{F -\mathrm{Id}}{2}$ are of rank $2$.  In fact, in the basis $\{e_1, e_2 = \Chi e_1, \xi, \nu\}$, the matrix  $\frac{F +\mathrm{Id}}{2}$ can be written as

\[  \frac 12 \left( \begin{array}{ccccc}
-h+1 & 0 & (V, e_2) &  (V, e_1) \\
0 & -h+1 & (V, e_1)& (V, e_2) \\
(V, e_2) & -(V, e_1) & h+1 & 0 \\

(V, e_1) & (V, e_2) & 0 & h+1 
\end{array} \right)\] 
Using that $h^2 + \Vert V \Vert^2 = 1$, one can check that it is of rank $2$. Same holds for $\frac{F -\mathrm{Id}}{2}$.
\end{proof}
\begin{rem}
Before giving some applications, we want to mention that both equivalent assertions of Theorem \ref{thmCM} are also equivalent to a third one described in terms of the Dirac operators $D^1$ and $D^2$, and the energy-momentum  tensors associated to $\varphi_1$ and $\varphi_2$. 
We recall that the energy-momentum tensors $Q_{\varphi_j}$, $j=1,2$, associated to the spinors field $\varphi_j$ are the $(2,0)$-tensors respectively defined by 
$$Q_{\varphi_j}(X,Y)=\frac{1}{2}\pre(\gamma_j(X)\cdot\nabla^j_{Y}\varphi+\gamma_j(Y)\cdot\nabla^j_{X}\varphi,\frac{\varphi}{|\varphi|^2}).$$
This third assertion can be written as:\\\\
{\it 3. There exists 2 $\spinc$ structures on $M$ carrying each one a  non-trivial spinor $\varphi_1$ and $\varphi_2$ of constant norms and satisfying
$$D^1\varphi_1=\frac{3}{2}H\varphi_1\ \ \ \text{and} \ \ \ \gamma_1(\xi)\varphi_1= - i\varphi_1,$$
$$D^2\varphi_2=-\frac{3}{2}H\varphi_2\ \ \ \text{and} \ \ \ \gamma_2(V)  \varphi_2 = -i \gamma_2(\xi) \varphi_2 + h \varphi_2,$$
so that their energy-momentum tensors $Q_{\varphi_1}$ and $Q_{\varphi_2}$ are the same. Moreover, the curvature 2-form of the connection on  the auxiliary bundle associated with these two  $\spinc$ structure are given by ($j \in \{1, 2\}$) 
$$\left\{
\begin{array}{l}
\Omega^j(e_1, e_2) = \frac 12 (-1)^{j-1}c_1 (h-1) -\frac 12 c_2 (h+1),\\
\Omega^j(e_1, \xi) = \frac 12 \Big( (-1)^{j-1} c_1 - c_2\Big) (e_1, V),\\
\Omega^j (e_2, \xi) = \frac 12 \Big( (-1)^{j-1} c_1 - c_2\Big) (e_2, V),
\end{array}
\right.
$$
in the basis $\{e_1,e_2=\Chi e_1,e_3=\xi\}$.}\\

Indeed, clearly, the second assertion of Theorem \ref{thmCM} implies assertion $3.$. Reciprocally, as proven in \cite{LR}, $D^1\varphi_1=\frac{3}{2}H\varphi_1$ with $\varphi_1$ of constant norm implies that $\nabla^1_X\varphi_1=-\frac{1}{2}\gamma_1(E_1X)\varphi_1$ with $E_1=Q_{\varphi_1}$. Similarly, we also get $\nabla^2_X\varphi_2=\frac{1}{2}\gamma_1(E_2X)\varphi_2$ with $E_2=Q_{\varphi_2}$. Now, since $Q_{\varphi_1} = Q_{\varphi_2}$, this gives assertion 2 of Theorem \ref{thmCM}. 
\end{rem}

\section{Totally geodesic and totally umbilical hypersurfaces of $\mathbb M_{1} (c_1) \times \mathbb M_2(c_2)$}
In this section, we use our main result, Theorem \ref{thmCM}, to give some geometric results on totally geodesic and umbilical hypersurfaces of $\mathbb M_{1} (c_1) \times \mathbb M_2(c_2)$.

\begin{lem}
Let $\big(M^3, g=(.,.)\big)$ be a totally umbilical hypersurface of $\mathbb M_{1} (c_1) \times \mathbb M_2(c_2)$. Then,
\begin{eqnarray}\label{meanV}
\Vert V \Vert \vert c_1 -c_2\vert = 4 \Vert dH \Vert
\end{eqnarray}

\end{lem}

{\bf Proof.} From Theorem \ref{thmCM}, we know that there exists 2 $\spinc$ structures on $M$ carrying each one a  non-trivial spinor $\varphi_1$ and $\varphi_2$  satisfying
$$\nabla^1_X\varphi_1=-\frac{1}{2}\gamma_1(EX)\varphi_1 = -\frac H2 \gamma_1(X)\varphi_1\ \ \ \text{and}\ \ \  \gamma_1(\xi)\varphi_1= - i\varphi_1.$$
$$\nabla^2_X\varphi_2=\frac{1}{2}\gamma_2(EX)\varphi_2= \frac H2 \gamma_2(X)\varphi_2\ \ \ \text{and}\ \ \ \ \gamma_2(V)  \varphi_2 = -i \gamma_2(\xi) \varphi_2 + h \varphi_2.$$
The curvature 2-form of the connection on  the auxiliary bundle associated with these two  $\spinc$ structure are given by ($j \in \{1, 2\}$) 
\begin{eqnarray}\label{Ome1}
\Omega^j(e_1, e_2) = \frac 12 (-1)^{j-1}c_1 (h-1) -\frac 12 c_2 (h+1),
\end{eqnarray}
\begin{eqnarray}\Omega^i (e_1, \xi) = \frac 12 \Big( (-1)^{j-1} c_1 - c_2\Big) (e_1, V),\label{omega2}\end{eqnarray}
\begin{eqnarray}\Omega^j (e_2, \xi) = \frac 12 \Big( (-1)^{j-1} c_1 - c_2\Big) (e_2, V),
\label{Ome3}\end{eqnarray}
in the basis $\{e_1,e_2=\Chi e_1,e_3=\xi\}$.\\\\
{\bf For the second Spin$^c$ structure:} The Ricci identity for $M$ can be written as 
$$\frac 12 \gamma_2(\mathrm{Ric}^M (X)) \varphi_2 - \frac i2 \gamma_2(X \lrcorner \Omega^2) \varphi_2 = \frac 12 \gamma_2 (dH)\gamma_2(X)\varphi_2 +\frac 32 dH(X) \varphi_2 + H^2 \gamma_2(X) \varphi_2 .$$
For $X =\xi$, the \ of the scalar product of the previous identity with $\varphi_2$ gives 
\begin{eqnarray*}
-\frac i2 \Omega^2 (\xi, e_1) (\gamma_2(e_1)\varphi_2, \varphi_2) &-& \frac i2 \Omega^2 (\xi, e_2) (\gamma_2(e_2)\varphi_2, \varphi_2) \\&=& \frac 12 \pre(\gamma_2 (dH)\gamma_2(\xi)\varphi_2, \varphi_2)+ \frac 32 dH(\xi)
\end{eqnarray*}
Using that $-\gamma_2(e_1)\gamma_2(e_2)\gamma_2(\xi)\varphi_2 = \varphi_2$, we get 

$$-\frac i2 \Omega^2 (\xi, e_1) (\gamma_2(e_1)\varphi_2, \varphi_2) - \frac i2 \Omega^2 (\xi, e_2) (\gamma_2(e_2)\varphi_2, \varphi_2) =  dH(\xi)$$
Finally, using (\ref{Ide2}), (\ref{Ide3}), (\ref{omega2}) and (\ref{Ome3}), we obtain  $dH(\xi) = 0$. In a similar way, for $X = e_1$ the real part of the scalar product with $\varphi_2$ of the Ricci identity gives 
\begin{eqnarray*}
-\frac i2 \Omega^2 (e_1, e_2) (\gamma_2(e_2)\varphi_2, \varphi_2) &-& \frac i2 \Omega^2 (e_1, \xi) (\gamma_2(\xi)\varphi_2, \varphi_2) \\&=& \frac 12 \pre(\gamma_2 (dH)\gamma_2(e_1)\varphi_2, \varphi_2) + \frac 32 dH(e_1).
\end{eqnarray*}
Using that $-\gamma_2(e_1)\gamma_2(e_2)\gamma_2(\xi)\varphi_2 = \varphi_2$ and $dH(\xi) = 0$, we get 
\begin{eqnarray*}
-\frac i2 \Omega^2 (e_1, e_2) (\gamma_2(e_2)\varphi_2, \varphi_2) - \frac i2 \Omega^2 (e_1, \xi) (\gamma_2(\xi)\varphi_2, \varphi_2) = dH(e_1).
\end{eqnarray*}
Finally, using (\ref{Ide2}), (\ref{Ide3}), (\ref{omega2}) and (\ref{Ome1}), we obtain  $dH(e_1) =  \frac {c_1-c_2}{4} g(V, e_1) $. In a similar way we can get $dH(e_2) =  \frac {c_1-c_2}{4} g(V, e_2) $. Hence, we have $\Vert dH\Vert^2 = \frac{(c_1-c_2)^2}{16}\Vert V \Vert^2.$ For consistency, one can also take the first Spin$^c$ structure and check that a similar identity can be obtained.  
\begin{prop}
\label{pr1}
Let $M$ be a totally umbilical hypersurface of in $\mathbb M_{1} (c_1) \times \mathbb M_1(c_1)$. Then $M$ is totally geodesic or an extrinsic hypersphere. Moreover, if $c_1 \neq 0$, the universal cover of $M$ is a Non-Einstein Sasaki manifold or a product of a K\"ahler manifold (of complex dimension 1) with $\mathbb R$. If $c_1 = 0$, then $M$ is a $\spin$ manifold with a parallel or Killing spin spinor.  
\end{prop}
\begin{proof}Let $M$ be a totally umbilical hypersurface of $\mathbb M_{1} (c_1) \times \mathbb M_1(c_1)$. We have from (\ref{meanV}) that $dH = 0$, so $H$ is constant. Assume that  $c_1 \neq 0$. If this constant $H$ is 0, $M$ has a parallel Spin$^c$ spinor and if $H \neq 0$, then $M$ has a Killing Spin$^c$ spinor. Form the classification of parallel and Killing Spin$^c$ spinors \cite{Moro1}, we get the desired result. If $c_1 =0$, then the curvature of the auxiliary line bundle defining the Spin$^c$ structure is zero and hence $M$ is a $\spin$ manifold with parallel or Killing spin spinor.
\end{proof}
\begin{prop}
\label{pr2}
Let M be a totally umbilical hypersurface of in $\mathbb M_{1} (c_1) \times \mathbb M_2(c_2)$ ($c_1 \neq c_2$) having a local product structure. Then $M$ is totally geodesic or an extrinsic hypersphere. If $c_1 \neq c_2\neq 0$. the universal cover of $M$ is a non-Einstein Sasaki manifold or a product of a K\"ahler manifold (of complex dimension $1$) with $\mathbb R$
\end{prop}
\begin{proof} Let $M$ be a totally umbilical hypersurface of $\mathbb M_{1} (c_1) \times \mathbb M_2(c_2)$. Since $V= 0$, we have from (\ref{meanV}) that $dH = 0$, so $H$ is constant. If this constant is $0$, $M$ has a parallel Spin$^c$ spinor and if $H \neq 0$, then $M$ has a Killing Spin$^c$ spinor. From the classification of parallel and Killing Spin$^c$ spinors \cite{Moro1}, we get the desired result. 
\end{proof}
Using also Theorem \ref{thmCM}, one can also prove the following:
\begin{prop}\label{pr3}
Simply connected 3-dimensional homogeneous manifolds $\Ekt$ ($\tau \neq 0$), with 4-dimensional isometry group cannot be immersed in $\mathbb M_{1} (c_1) \times \mathbb M_1(c_1)$ as totally umbilical hypersurfaces.
\end{prop}
\begin{proof}
For $c_1=0$, this has been proved by Lawn and Roth \cite{LR}, even without assuming the umbilicity. Assume that $c_1 \neq 0$ and $\Ekt$ can be immersed in a totally umbilical way in 
$\mathbb M_{1} (c_1) \times \mathbb M_1(c_1)$. By Proposition \ref{pr1}, we have that $H$ is constant and by Theorm \ref{thmCM}, $\Ekt$ has two $\Spinc$ structures 
carrying each one a  non-trivial spinor $\varphi_1$ and $\varphi_2$  satisfying
$$\nabla^1_X\varphi_1=-\frac{H}{2}\gamma_1(X)\varphi_1\ \ \ \text{and}\ \ \  \gamma_1(\xi)\varphi_1= - i\varphi_1.$$
$$\nabla^2_X\varphi_2=\frac{H}{2}\gamma_2(X)\varphi_2\ \ \ \text{and}\ \ \ \ \gamma_2(V)  \varphi_2 = -i \gamma_2(\xi) \varphi_2 + h \varphi_2.$$
The curvature 2-form of the connection on  the auxiliary bundle associated with these two  $\spinc$ structure are given by ($j \in \{1, 2\}$) 
$$\Omega^j(e_1, e_2) = \frac 12 (-1)^{j-1}c_1 (h-1) -\frac 12 c_2 (h+1),$$
$$\Omega^i (e_1, \xi) = \frac 12 \Big( (-1)^{j-1} c_1 - c_2\Big) (e_1, V),$$
$$\Omega^j (e_2, \xi) = \frac 12 \Big( (-1)^{j-1} c_1 - c_2\Big) (e_2, V),$$
in the basis $\{e_1,e_2=\Chi e_1,e_3=\xi\}$. We will call the first one $\spinc$ structure $t_1$ and the second one $t_2$.
Since $H$ is constant, these two spinors are in fact real Killing spinors. But, it is known \cite{NR12} that the manifold $\Ekt$ has only two $\Spinc$ structures carrying Killing spinors. The first one (call it $t_3$) carries a Killing spinor $\varphi$ with Killing constant $\frac {\tau}{2}$ and for which $\gamma_3(\xi)\varphi = -i \varphi, \Omega^3(e_1, e_2) = -(\kappa-4\tau^2)$ and $\xi\lrcorner\Omega^3 = 0$, where we denote by $\gamma_3$ and $i\Omega^3$ the Clifford multiplication and the curvature 2-form of the auxiliary line bundle associated to the structure $t_3$. The second one (let's call it $t_4$) also carries a Killing spinor $\varphi$ with Killing constant $\frac{\tau}{2}$ for which $\gamma_4(\xi)\varphi = i \varphi, \Omega^4(e_1, e_2) = (\kappa-4\tau^2)$ and $\xi\lrcorner\Omega^4 = 0$, where we denote by $\gamma_4$ and $i\Omega^4$ the Clifford multiplication and the curvature 2-form of the auxiliary line bundle associated to the structure $t_4$. By comparison, we must have that $t_1 = t_3$ and $t_2 = t_4$. Thus we get 
$$\gamma_2(V)  \varphi_2 = -i \gamma_2(\xi) \varphi_2 + h \varphi_2 = (h+1)\varphi_2.$$
Hence $V=0$. This means that $\Ekt$ has a local product, which is a contradiction.
\end{proof}
{\bf Acknowledgement:} 
This work was initiated in 2017 during the one month research stay of the first author at the ``Laboratoire d'Analyse et de Math\'ematiques Appliqu\'ees" (UMR 8050) of the University of  Paris-Est Marne-la-Vall\'ee. The first author gratefully acknowledges the support  and hospitality of the  University of  Paris-Est Marne-la-Vall\'ee.

\end{document}